\documentclass{article}
\usepackage{amsmath, amsthm, amssymb, fullpage, xypic}
\usepackage[pdftex]{graphicx}
\usepackage[all, knot]{xy}
\usepackage{multirow}
\usepackage{hyperref}
  \hypersetup{colorlinks=true,citecolor=blue}

\xyoption{arc}
\xyoption{poly}
\xyoption{curve}
\xyoption{arrow}

\newtheorem{theorem}{Theorem}[section]
\newtheorem*{thm*}{Theorem}
\newtheorem{thm}{Theorem}
\newtheorem{lemma}[theorem]{Lemma}
\newtheorem{prop}[theorem]{Proposition}
\newtheorem{defn}[theorem]{Definition} 
\newtheorem{dfn}[theorem]{Definition} 

\newtheorem{cor}[theorem]{Corollary}

\newtheorem{rmk}[theorem]{Remark}

\newcommand{\lra}{\longrightarrow}
\newcommand{\lla}{\longleftarrow}

\newcommand{\sr}{\stackrel}

\newcommand {\F} {\mathbb{F}}

\newcommand {\bS} {\mathbb{S}}
\newcommand {\K} {\mathbb{K}}

\newcommand {\PP} {\mathbb{P}}
\newcommand {\Z} {\mathbb{Z}}
\newcommand {\R} {\mathbb{R}}

\newcommand {\sus} {\Sigma}
\newcommand {\tensor} {\otimes}
\newcommand {\iso} {\cong}
\newcommand {\dsum} {\oplus}
\newcommand {\we} {\simeq}
\newcommand {\Ei} {E_{\infty}}
\newcommand {\Hi} {H_{\infty}}
\newcommand {\xt} {\widetilde{x}}

\newcommand {\kk} {\mathrm{k}}

\newcommand {\oo} {\overline}
\newcommand {\ssra} {\Rightarrow}
\newcommand {\oxi} {\overline{\xi}}
\newcommand {\BP} {BP\langle 2 \rangle}

\newcommand {\bu} {\bullet}
\newcommand {\mc} {\mathcal}
\newcommand {\HF} {\mathrm{H}\mathbb{F}}
\newcommand {\Hk} {\mathrm{H}\mathrm{k}}
\newcommand {\HH} {\mathrm{H}}
\newcommand {\HZ} {\mathrm{H}\mathbb{Z}}

\newcommand {\tGamma}{\widetilde{\Gamma}}
\newcommand{\bss}{\begin{singlespace}}
\newcommand{\ess}{\end{singlespace}}
\newcommand{\Ab}{A_{\bullet}}
\newcommand{\Bb}{B_{\bullet}}
\newcommand{\Xb}{X_{\bullet}}
\newcommand{\Yb}{Y_{\bullet}}
\newcommand{\Gb}{\Gamma_{\bullet}}
\newcommand{\tGb}{\tGamma_{\bullet}}

\DeclareMathOperator{\Tor}{\mathrm{Tor}}
\DeclareMathOperator{\colim}{\mathrm{colim}}
\DeclareMathOperator{\hocolim}{\mathrm{hocolim}}

\DeclareMathOperator{\Hom}{\mathrm{Hom}}
\DeclareMathOperator{\Cof}{\mathrm{Cof}}

\title{Power operations in the K\"unneth spectral sequence and commutative $\HF_p$-algebras}
\author{Sean Tilson}

\begin{document}
\maketitle

\begin{abstract}
In this paper, we prove the multiplicativity of the K\"unneth spectral sequence.
This is established by an analogue of the Comparison Theorem from homological algebra, which we suspect may be useful for other spectral sequences.
This multiplicativity is then used to compute the action of the Dyer-Lashof algebra on $\HF_p \wedge_{ku} \HF_p$, $\HF_p \wedge_{\BP} \HF_p$, and part of the action on $\HF_p \wedge_{MU} \HF_p$.
We then relate these computations to the construction of commutative $R$-algebra structures on commutative $\HF_p$-algebras.
In the case of $MU$, we obtain a necessary closure condition on ideals $I\subset MU_*$ such that $MU/I$ can be realized as a commutative $MU$-algebra.
\end{abstract}

\tableofcontents
\section{Introduction}

In modern algebraic topology, the theme of importing classical constructions in algebra has been very fruitful.
Modern models of spectra have enabled topologists to exploit the intuition gained from homological and commutative algebra in a rigorous fashion.
The goal of this paper is to continue in that vein and investigate structure present on relative smash products of commutative ring spectra.
The relative smash product plays the role of the (derived) relative tensor product and so is a natural source of new examples of module and ring spectra.
Our primary focus will be on understanding these relative smash products as commutative ring spectra.
In particular, we will find that $A \wedge_R A$ contains interesting information about $R$ as a commutative ring spectrum.

The K\"unneth spectral sequence (sometimes abbreviated as KSS) is the main tool for computing invariants of relative smash products.
It is a natural extension of the classical K\"unneth Theorem of algebraic topology which gives the short exact sequence
\[
0\lra \bigoplus_{i+j=n}\HH_i(X;\Z)\tensor \HH_j(Y;\Z) \lra \HH_n(X \times Y; \Z) \lra \Tor^{\Z}_1(\HH_i(X;\Z),\HH_j(Y;\Z)) \lra 0
\]
where $X$ and $Y$ are spaces and $\HH_n(-;\Z)$ is singular homology with coefficients in $\Z$.
As is common in algebraic topology, this short exact sequence was generalized to a spectral sequence for arbitrary homology theories so that we can recover the above Theorem from the following spectral sequence.
In Chapter 4 section 6 of \cite{EKMM}, they construct the K\"unneth spectral sequence
\[
E_2^{s,t}= \Tor^{\pi_* R}_s(\pi_*A,\pi_*B)_t \Longrightarrow \pi_{s+t}(A \wedge_R B).
\]
where $R$ is an $S$-algebra, and $A,B$ are right and left $R$-modules respectively.
The K\"unneth Theorem is a particular case of the above where $R=\HH\Z$, $A=\HH\Z\wedge \sus^{\infty}X_+$, and $B=\HH\Z\wedge \sus^{\infty}Y_+$.
Previous work on the multiplicativity of the Eilenberg-Moore spectral sequence was done by \cite{LigaardMadsen}.
While they use the diagonal map of spaces in order to developed power operations and used their results to compute $\HH_*(B(G/Top))$, we are unable to take advantage of such structure as it is specific to spaces.
Our work here is inspired by that of Bruner in \cite{HRS} on the Adams spectral sequence as well as the work of Baker and Richter in \cite{BakerRichterAdams} on the K\"unneth spectral sequence.

The main aim of this paper is to enhance this spectral sequence so that we can understand the multiplicative structure present on relative smash products of ring spectra.
In particular, the multiplicativity of the above spectral sequence can provide collapse results, which are very useful in aiding computations.
We establish the multiplicativity of the KSS by use of our Comparison Theorem \ref{thm: Comparison Theorem}.
\begin{thm*}[Comparison Theorem]
\label{thm: comparison intro}
Suppose that we have a map $f: Y \to A$ of $R$-module spectra, an exact filtration $A_{\bullet} \subset A$, and a free and exhaustive filtration $Y_{\bullet} \subset Y$.
Also, suppose that there exist $f_{-1}: Y_{-1} \to A_{-1}$ such that
\[
\xymatrix{
Y_{-1} \ar[rr] \ar[dd]_{f_{-1}}&
&
Y \ar[dd]^f\\
&
&
\\
A_{-1} \ar[rr]&
&
A
}
\]
commutes.
Then there is a map of filtrations $Y_i \sr{f_i} \lra A_i$ such that $\colim f_i\we f$ under the equivalences $\colim Y_i \we Y$ and $\colim A_i \we A$.
Furthermore, the lift $f_{\bu}$ of $f$ is unique up to homotopy of filtered modules, in the sense of Definition \ref{defn:homotopy of filtered modules}.
\end{thm*} 
This result can be seen as both a lift of the classical Comparison Theorem of homological algebra to filtered $R$-module spectra and as a cellular approximation Theorem.
Using this result we can derive the multiplicativity of the KSS.
This multiplicativity is precisely an example of lifting properties of an $R$-module to a filtration of that $R$-module.
\begin{thm*}[Multiplicativity of the K\"unneth spectral sequence]
\label{thm: mult kss intro}
Let $R$ be a commutative $S$-algebra, and $A$ and $B$ be $R$-algebras.
The K\"unneth spectral sequence is multiplicative.
Further, the filtration of the KSS has an $\Hi$-structure that filters the commutative $S$-algebra structure of $A\wedge_R B$.
\end{thm*}
This $\Hi$-structure is not simply gotten by having each stage of the filtration being an $\Hi$-algebra.
It is more subtle and not directly accessible by algebraic techniques, see Section \ref{subsec:relation to alg} as well as Remark \ref{rmk: BSS}.
Note that this result fixes an error in the literature.
In \cite{BakerLazarev}, the authors provide a flawed proof of the multiplicativity of the KSS, see subsection \ref{subsubsec:previous work} for more details.
With this new-found multiplicative structure we are able to compute the homotopy of $\HF_p \wedge_R \HF_p$ for various commutative $S$-algebras $R$.
We then take advantage of the additional $\Hi$-structure to compute the action of the Dyer-Lashof algebra on various relative smash products.
For example, we have the following computation.
\begin{prop}
The K\"unneth spectral sequence 
\[
\Tor^{MU_*}_* (\F_p,\F_p) \ssra \pi_*\HF_p \wedge_{MU} \HF_p
\]
collapses at the $E_2$ page.
The resulting algebra $\pi_* \HF_p \wedge_{MU} \HF_p$ is an exterior algebra $E_{\F_p}[\oo{p},\oo{x}_1,\oo{x}_2, \ldots]$ with Dyer-Lashof operations given by $Q^{\rho(i)}(\oo{p})=\pm\oo{x}_{p^{i-1}-1}$ and $Q^{p^i}(\oo{x}_{p^{i-1}-1})=\pm\oo{x}_{p^i-1}$ where $\vert \oo{p} \vert=1$ and $\vert \oo{x}_n \vert=2n+1$.
\end{prop}
The classes $x_{p^i-1}\in\pi_{2p^i-2}MU$ are classes that hit the Hazewinkel generators under the $p$-localization map
\[
MU \lra MU_{(p)}.
\]
The classes $\oo{x_i}\in\pi_{2i+1}(\HF_p\wedge_{MU}\HF_p)$ are induced classes, see Proposition \ref{prop: Tor-1 realization} for an explicit construction.
This result gives the following necessary condition on an ideal in order for the quotient to be realizable as a commutative $S$-algebra.
This is result is comparable to the work of Strickland in \cite{StricklandMU}.
\begin{cor}
\label{cor:strickland analog at p}
Let $I$ be an ideal of $MU_*$ generated by a regular sequence. 
If $I$ contains a finite non-zero number of the $x_{p^i-1}$, then the quotient map $MU \to MU/I$ cannot be realized as a map of commutative $S$-algebras.
\end{cor}

This result is not the only application of our computations.
The classes in these relative smash products can frequently be realized as differences of null-homotopies so that one obtains a function
\[
\xymatrix{
\pi_n(A\wedge_R A)\times CAlg_{R}(A,X)^{\times 2} \ar[r]&
\pi_nX\\
(\oo{x},f,g) \ar@{|->}[r]&
d(f,g)_*(\oo{x})
}
\]
that measures differences between commutative $R$-algebra maps $f$ and $g$.
The map 
\[
d(f,g):A\wedge_R A \lra X
\]
is the coproduct of $f,g\in CAlg_R(A,X)$, see section \ref{sec:applications and interpretations} .
It is a map of commutative ring spectra and therefore respects power operations.
So if we think of $\pi_*(A \wedge_R A)$ as giving functions on $CAlg_R(A,-)$ we can then hope to interpret the power operations we compute as giving dependence relations between such functions.
We believe this perspective will be fruitful and hope to study it in future work with Markus Spitzweck on fundamental groups of derived affine schemes.
We use the multiplicativity of the KSS to compute the homology of connective Morava $E$-theory with Lukas Katth\"an in \cite{HFpenKatthaenTilson}.

\subsection{Outline}
We begin in section \ref{sec:Filtrations and Complexes} by introducing some basic properties of the category of filtered $R$-modules and the category of complexes of $R$-modules for a commutative $S$-algebra $R$.
The category of filtered $R$-modules, its tensoring over spaces, and its monoidal structure are what we will take advantage of in order to obtain structural results about spectral sequences.
We will explain the relationship between the categories of filtered $R$-modules and complexes of $R$-modules as well as relating their monoidal structures.
That structures on filtrations induce analogous structures on spectral sequences can be found in \cite{DuggerMultSS} or \cite{TilsonThesis}.
We neglect to include it here as it is well known to the experts.

In section \ref{sec:Comparison Theorem}, we prove our Comparison Theorem \ref{thm: comparison intro}.
This result is what allows us to show that the KSS is multiplicative.
This result is essentially a cellular approximation argument.
We think it may be of use in other settings, just as the Comparison Theorem of homological algebra has many applications.

In section \ref{sec:KSS}, we begin studying the KSS.
This is the main tool for computing invariants of relative smash products.
Using the results of the previous section, we are able to show that many examples of KSSs are multiplicative.
This fact should be thought of as a mirror of the analogous result regarding the Adams and Adams-Novikov spectral sequences.
After proving that the KSS has an $\Hi$-filtration, Theorem \ref{thm: mult kss intro}, we utilize this extra structure to compute the action of the Dyer-Lashof algebra on various relative smash products of $\HF_p$ in section \ref{sec:comps}.
While we only discuss relative smash products over $ku$, $\BP$, and $MU$ here, we hope to discuss relative smash products over less regular (in the technical sense) ring spectra in the future.

In section \ref{sec:applications and interpretations}, we interpret our computations in terms of differences of null-homotopies.
This interpretation yields a necessary condition for determining if a map between commutative $S$-algebras over $\HF_p$ is a map of commutative $S$-algebras in terms of its effect on homotopy.
Strickland \cite{StricklandMU} obtained comparable, and stronger, results.

We expect similar results when one works over more complicated commutative $S$-algebras than $\HF_p$.
However, our methods here do not extend to Landweber exact spectra.
This is due to the observation of Gerd Laures that Landweber exact spectra are flat over $MU$, which was related to us by Charles Rezk.

\subsection{Conventions}
We work with $S$-modules as developed in \cite{EKMM}.
The main reason for this choice is that the model category structures developed in \cite{EKMM} is well adapted to computation and categorical manipulations.
In particular, every $S$-module is fibrant and there exist functorial factorization and cofibrant replacement functors.
Thus, when we restrict attention to cofibrant objects and consider diagrams of cofibrations we can work as if in the homotopy category with the added benefit that our constructions are natural.
This will be obvious from our work in section \ref{sec:Filtrations and Complexes} as our arguments feel like those made in triangulated categories but are a bit stronger.
We will also not distinguish between fiber and cofiber sequences as we are working in a stable model category.

We will frequently go back and forth between $R$-module spectra and their graded homotopy groups.
If we have denoted an $R$-modules spectrum by $F_i$ then by $\F_i$ we will mean $\pi_*F_i$ as a $\pi_*R$-module.
We will reserve the letter $F$ for free $R$-modules, and $A$ and $B$ for commutative $S$-algebras.
By a free $R$-module $F$ we mean a wedge of $S^n_R$, which are cofibrant replacements of $\sus^nR$ as an $R$-module.
This is necessary as one drawback of the model of spectra presented in \cite{EKMM} is that $R$ may not be a cofibrant $R$-module.
Our primary reason for working with free $R$-modules is that $R$-module maps out of free $R$-modules are determined by their effect in homotopy.
Most importantly, if a map out of a free $R$-module is $0$ in homotopy then it is null-homotopic.

\subsection{Relation to other work}
There has been previous work on the K\"unneth spectral sequence, its multiplicativity, and related spectral sequences.
In \cite{BakerLazarev}, a proof is given that the K\"unneth spectral sequence is multiplicative.
Unfortunately, this proof has a fatal flaw, they use another filtration that they claim is equivalent.
We address this mistake in Subsection \ref{subsubsec:previous work}.
There is also the work of Ligaard and Madsen on power operations for the Eilenberg-Moore spectral sequence in the case of infinite loop spaces.
Key to there construction of operations is the existence of a diagonal map on spaces.
In our setting we have no such diagonals to take advantage of.

There is the B\"okstedt spectral sequence which looks very similar to the K\"unneth spectral sequence.
However, the filtration on the former comes from, when one is working with $\Ei$-ring spectra, the skeletal filtration of $S^1$ inducing a filtration on $S^1 \tensor A$.
This is then a filtration of $\Ei$-ring spectra and the fundamental class of the circle acts via $\Ei$-ring maps, hence $\sigma$ commutes with the action of the Dyer-Lashof algebra.
This comes from the action of $S^1$ which we don't have on $k\wedge_R k$.
See Remark \ref{rmk: BSS} for more details.
Our computations of $k \wedge_R k$ can not be written as tensoring with $S^1$.
Angeltveit and Rognes take great advantage of this circle action in \cite{AngeltveitRognes} to provide a rigorous proof of B\"okstedts result regarding operations in this spectral sequence.

There is also the spectral sequence coming from filtering the Bar construction $B(A,R,B)$ by its skeleta.
In the case that $R$ is a $\Ei$-algebra and both $A$ and $B$ are $R$-algebras, this also has a product structure.
However, the $E_2$-page is not the same as applying the functor $\pi_*(R\wedge -)$ does not always produce free $\pi_*R$-modules and so applying homotopy to this bar construction does not provide a free resolution.
This is reminiscent of the fact that tensoring $M$ with a free module does not produce a free module (consider the situation when $M$ has a nontrivial torsion submodule).

Having a theory of operations that always preserves the filtration would imply that the composition
\[
D^{(k)}_r(A_n) \lra \Cup_{i+j=k+n} \tGamma^r_{k,n} \lra A_{k+n}
\]
actually lifts to a map
\[
D^{(k)}_r(A_n) \lra A_n
\]
for a filtered spectrum 
\[
\Ab: \cdots \lra A_n \lra A_{n+1} \lra A_{n+2} \lra \cdots.
\]
In particular, when $k=0$ we would have a map
\[
A_n^{\wedge r} \lra A_n.
\]
In our computations in Section \ref{sec:comps}, we see that this is not the case as the induced product on $\Tor$ is that of a graded algebra.
If one considers $\Tor^{\F_2[x]/(x^2)}(\F_2,\F_2)$ where $x$ is in degree $1$ then we get a graded algebra, in fact a divided power algebra.
The power operations we are able to compute here in Section \ref{sec:comps} do in fact all preserve the filtration.
Why this is the case is open to further investigation.
We certainly do not believe it to be necessarily true.

\section*{Acknowledgments}
There are many people who deserve thanks and acknowledgment for their support, criticisms, and advice during the production of this work.
However, I will be brief as a more thorough list of acknowledgements can be found in my thesis.
This document contains of much of the work I did for my thesis at Wayne State University under the guidance of Robert Bruner.
As such, I must thank him at the outset for being right about things so frequently and being patient when I didn't believe him.
I would also like to thank the members of my committee Daniel Isaksen, John Klein, James McClure, and Andrew Salch.
In particular, I learned a great deal from Andrew in the last couple years of my program.
Michael Mandell suggested what became the Comparison Theorem \ref{thm: Comparison Theorem}, and for that I am grateful.
I would similarly like to thank Andrew Baker and Tyler Lawson for their encouragement, insight, and interest in various aspects of this project.
Lastly, this would have been impossible without the support of Jim Veneri and Mike Catanzaro.
I would also like to thank an anonymous referee for helpful comments.

\section{Filtrations and Complexes}
\label{sec:Filtrations and Complexes}
In this section we develop some preliminary machinery and terminology for filtered $R$-module spectra, referred to frequently as filtrations.
The goal of this section is the necessary background for our Comparison Theorem, Theorem \ref{thm: Comparison Theorem}, and working with spectral sequences.
We believe this section may also be of use to those wishing to establish similar structures on other spectral sequences.

The main objects discussed in this section are filtered $R$-modules and complexes of $R$-modules.
There is a correspondence between these two categories, but it is only partially defined on complexes of $R$-modules.
While we can always associate the associated graded complex to any filtration, we are not always able to invert this procedure.
This is possible when the complex is a free resolution, in the classical sense, which will be proved in section \ref{subsec:filtration to resolution and back again}.

\subsection{Basic definitions}
Let $R$ be a commutative $S$-algebra that is cofibrant.
We begin by giving a few definitions.
\begin{dfn}
A \underline{filtered $R$-module} $\Ab$ is a diagram of $R$-modules

\[
	\xymatrix{
	A_{-1}=* \ar[r]^{\alpha_{-1}}&
     A_0 \ar[r]^{\alpha_0}&
	A_1 \ar[r]^{\alpha_1}&
     A_2 \ar[r]&
	\cdots\\
}
\]
where each $\alpha_i$ is a cofibration of $R$-modules.
We call a filtered $R$-module $\Ab$ a \underline{filtration} of an $R$-module $A$ if for every $i$ there is a map $j_i:A_i \to A$ of $R$-modules such that

\[
	\xymatrix{
	A_{-1}=* \ar[r]^{\alpha_{-1}} \ar[drrrr]&
	A_0 \ar[r]^{\alpha_0} \ar[drrr]&
	A_1 \ar[r]^{\alpha_1} \ar[drr]&
	A_2 \ar[r]^{\alpha_2} \ar[dr]&
	\cdots \ar[r] \ar[d]&
	\colim(A_i) \ar[dl]\\
	&
	&
	&
	&
	A&
}
\]
commutes. 
We call the filtration \underline{free} or \underline{projective} if the filtration quotients, $\sus^i F_i := \Cof(A_{i-1} \lra A_i)$, are free $R$-modules or summands of free $R$-modules, respectively.
We call the filtration \underline{exhaustive} if the $j_i$ induce an equivalence $\hocolim_i A_i \approx A$.
Let $\sus^i K_i:= Fiber(A_i \sr {j_i} \lra A)$, and let $\pi_i$ denote maps induced by the $\alpha_i$ via the following diagram.
\[
	\xymatrix{
	\sus^i K_i \ar[rr] \ar[d]_{\pi_i}&
	&
  A_i \ar[rr]^{j_i} \ar[d]_{\alpha_i}&
	&
	A \ar[d]^{1_A} \\
	\sus^{i+1} K_{i+1} \ar[rr]&
	&
	A_{i+1} \ar[rr]^{j_{i+1}}&
	&
	A
}
\]
We call the filtration \underline{exact} if the maps $\pi_{i}$ are all 0 in homotopy.
We will call a filtration $\Ab$ \underline{cellular} if each $R$-module $A_i$ is a cellular $R$-module and each pair $(A_{i+1},A_i)$ is a relative cell $R$-module.
If there is a morphism of filtrations $f_{\bu}: \Ab\to \Bb$ such that each $f_i: A_i \to B_i$ is a weak equivalence and $\Bb$ is a cellular filtration, then we say that $\Ab$ has the homotopy type of a cellular filtration.
\end{dfn}

Every cofibrant $R$-module $A$ can be thought of as the filtered $R$-module $c(A)_{\bu}$ by taking all of the maps to be the identity.
The associated graded of this filtration is the complex of $R$-modules that is just $A$ in degree $0$ and the trivial $R$-module in every other degree.
Frequently, we will be considering filtrations $\Ab$ of $A:=\hocolim_{\bu} \Ab$ and so we will not mention the $R$-module that $\Ab$ is a filtration of, as it is implicit.
We will be primarily concerned with cellular filtrations.
As every $R$-module $M$ has a cellular filtration this will impose no real restriction.
This will follow from our construction of a projective and exact filtration from a projective resolution, see Proposition \ref{prop:resolution to filtration}.
Therefore, filtrations will be assumed to be cellular without comment.
As the maps in the filtration are assumed to be cofibrations $\colim \Ab \we \hocolim \Ab$.

\begin{lemma}
\label{lem:partial=0}
In an exact filtration, each map 
\[
\sus^{-1} A \sr{\partial_i} \lra \sus^i K_i,
\]
$i \geq 0$, induces the 0 map in homotopy.
Thus $\pi_* \sus^i K_i$ is the kernel of $j_i : \pi_* A_i \to \pi_* A$. 
Further, every exact filtration is exhaustive.
\end{lemma}

As mentioned above, we also work with complexes of $R$-module spectra.
\begin{dfn}
A \underline{complex} $A \lla F_*$ (of $R$-modules) is a diagram
\[
	\xymatrix{
	A&
	& 
	F_0 \ar[ll]_\varepsilon&
	& 
	F_1 \ar[ll]_{d_0}&
	& 
	F_2 \ar[ll]_{d_1}&
	& 
	F_3 \ar[ll]_{d_2}& 
	\ar[l] \cdots \\
	    }
	    \]  
such that
\[
	\xymatrix{
	\pi_* A&
	& 
	\pi_* F_0 \ar@{->>}[ll]_\varepsilon&
	& 
	\pi_* F_1 \ar[ll]_{d_0}&
	& 
	\pi_* F_2 \ar[ll]_{d_1}&
	& 
	\pi_* F_3 \ar[ll]_{d_2}& 
	\ar[l] \cdots \\
	    }
	    \]
is a complex in the category of $\pi_* R$-modules.
A complex is \underline{exact} if taking homotopy groups induces an exact sequence in the category of $\pi_* R$-modules.
We will call such a complex a \underline{resolution}.
We call a resolution \underline{free} or \underline{projective} if the $F_i$ are free or \underline{projective} $R$-modules, respectively.
\end{dfn}

We realize this use of the asterisk may conflict with the convention that $M_*$ denotes $\pi_*M$, but the alternative is less desirable.

\begin{lemma}
Given a projective resolution of $\pi_*A$ by $\pi_*R$-modules
\[
	\xymatrix{
	\pi_* A&
	& 
	\F_0 \ar@{->>}[ll]_\varepsilon&
	& 
	\F_1 \ar[ll]_{d_0}&
	& 
	\F_2 \ar[ll]_{d_1}&
	& 
	\F_3 \ar[ll]_{d_2}& 
	\ar[l] \cdots \\
	    }
	    \]
there is a projective resolution of $R$-modules

\[
	\xymatrix{
	A&
	& 
	F_0 \ar@{->>}[ll]_\varepsilon&
	& 
	F_1 \ar[ll]_{d_0}&
	& 
	F_2 \ar[ll]_{d_1}&
	& 
	F_3 \ar[ll]_{d_2}& 
	\ar[l] \cdots \\
	    }
	    \]
that realizes the above algebraic resolution. 
\end{lemma}
\begin{proof}
First, we construct projective $R$-modules $F_i$ such that $\pi_*F_i=\F_i$ which is straightforward.
This can be done using that projective $R$-module spectra are summands of free $R$-module spectra.
To construct maps between the projective $R$-modules $F_i$ we use the fact that maps out of free $R$-modules exist whenever they exist in the homotopy category and then use the structure maps coming from the fact that the projective modules are summands of the free modules.
\end{proof}
The proof above will be indicative of later statements.
The proof in the case of free $R$-modules will imply the result in the case of projective modules by using the structure maps of the projective module in question as a summand of a free module.

Note that while we do not claim that an exact complex of $R$-modules is one such that $d^2$ is null-homotopic but that $\pi_*(d^2)=0$.
These two conditions do coincide when the complex is of projective $R$-modules.

\subsection{Filtration to Resolution and back again}
\label{subsec:filtration to resolution and back again}
Every filtration has a corresponding complex constructed as the ``associated graded'' of the filtration.
However, it is not possible to construct a filtration for any given complex.
In case that complex is a projective resolution, we can construct an associated filtration.
The associated filtration will be projective and exact, and have the resolution we began with as its associated graded complex.
The assumption that a filtration or complex be both projective and exact is a sufficient condition in order to have a correspondence between filtrations and complexes.
The only property of projective modules that we use is that a map out of a projective module is null-homotopic when it induces the 0 map in homotopy.
In section \ref{sec:Comparison Theorem}, we will extend this correspondence to morphisms.

Here we construct the associated graded complex of a filtration.
\begin{prop}
Given a projective and exact filtration  

\[
	\xymatrix{
	A_0 \ar[r]^{\alpha_0} \ar[drrr]&
	A_1 \ar[r]^{\alpha_1} \ar[drr]&
	A_2 \ar[r]^{\alpha_2} \ar[dr]&
	\cdots \ar[r] \ar[d]&
	\colim(A_i) \ar[dl]\\
	&
	&
	&
	A&
	\\
	}
	\]
there is an associated projective resolution
\[
	\xymatrix{
	A&
	& 
	F_0 \ar@{->>}[ll]_\varepsilon&
	& 
	F_1 \ar[ll]_{d_0}&
	& 
	F_2 \ar[ll]_{d_1}&
	& 
	F_3 \ar[ll]_{d_2}& 
	\ar[l] \cdots \\
	    }
	    \]
where $\Cof(A_{i-1} \lra A_i)=\sus^i F_i$.
\end{prop}

\begin{proof}
To construct the desired resolution, we first need to determine the cofiber of $\sus^{i} K_{i} \lra \sus^{i+1} K_{i+1}$, where $\sus^i K_i:= Fiber(A_i \sr {j_i} \lra A)$.
We make use of Verdier's axiom to identify this cofiber.
\[
	\xymatrix{
	A_i \ar@/^1pc/[rr]_{j_i} \ar[dr]_{\alpha_i}&
	& 
	A \ar@/^1pc/[rr]_{\partial_{i+1}} \ar[dr]^{\partial_i}&
	& 
	\sus^{i+2} K_{i+1} 
	\\
	&
	A_{i+1} \ar[dr]^{\beta_{i+1}} \ar[ur]_{j_{i+1}}& 
	&
	\sus^{i+1} K_i \ar[dr]^{q_i} \ar[ur]_{\pi_{i+1}}&
	\\
	\sus^{i+1} K_{i+1} \ar@/_1pc/[rr]^{\iota_{i+1}} \ar[ur]_{q_{i+1}}& 
	& 
	\sus^{i+1} F_{i+1} \ar[ur]_{p_{i+1}} \ar@/_1pc/[rr]& 
	& 
	\sus A_i 
}
\]
The cofiber sequences $\sus^{i+1} K_{i+1} \lra \sus^{i+1} F_{i+1} \lra \sus^{i+1} K_i$ induce short exact sequences in homotopy as $\sus^{i+1} K_i \lra \sus^{i+2} K_{i+1}$ is 0 in homotopy as the filtration is assumed to be exact. 
Now we splice together the short exact sequences $0 \lra \K_i \lra \F_i \lra \K_{i-1} \lra 0$ in the usual way to get 
	\[
	\xymatrix{
	A_*&
	& 
	\F_0 \ar@{->>}[ll]&
	& 
	\F_1 \ar@{->>}[dl] \ar[ll]&
	& 
	\F_2 \ar@{->>}[dl] \ar[ll]&
	& 
	\F_3 \ar@{->>}[dl] \ar[ll]& 
	\ar[l] \cdots \\
	&
	&
	& 
	\K_0 \ar@{>->}[ul]&
	&
	\K_1 \ar@{>->}[ul]&
	&
	\K_2 \ar@{>->}[ul]&
	   \\ }
	    \]               
an $R_*$-projective resolution of $A_*:=\pi_*{A}$, where $\F_i:= \pi_*(F_i)$ and $\K_i:= \pi_*(K_i)$.
As each $F_i$ is a projective $R$-module, this is enough to construct the desired projective resolution in $R$-modules.
\end{proof}

We now have the following realization result which will be used in the construction of the KSS.
It is equivalent to the construction in \cite{EKMM}.
\begin{prop}
\label{prop:resolution to filtration}
Given a projective resolution of $\pi_*A$ by $\pi_*R$-modules 
\[
	\xymatrix{
	A_*&
	& 
	\F_0 \ar@{->>}[ll]_\varepsilon&
	& 
	\F_1 \ar[ll]_{d_0}&
	& 
	\F_2 \ar[ll]_{d_1}&
	& 
	\F_3 \ar[ll]_{d_2}& 
	\ar[l] \cdots
}
\]
there is an exhaustive, exact, and projective filtration of $A$ by $R$-modules $A_i$
\[
	\xymatrix{
	A_{-1}=* \ar[r]^{\alpha_{-1}} \ar[drrrr]&
	A_0 \ar[r]^{\alpha_0} \ar[drrr]&
	A_1 \ar[r]^{\alpha_1} \ar[drr]&
	A_2 \ar[r]^{\alpha_2} \ar[dr]&
	\cdots \ar[r] \ar[d]&
	colim(A_i) \ar[dl]\\
	&
	&
	&
	&
	A&
}
\]
such that $\sus^i \F_i=\pi_* \Cof(A_{i-1} \to A_i)$. 	
\end{prop}

\begin{proof}
The resolution can be factored into SESs of $R_*$-modules 
\[
0 \lla \K_{i-1} \sr{p_{i-1}} \lla \F_{i} \sr{\iota_i} \lla \K_i \lla 0.
\]
\[
	\xymatrix{
	A_*&
	& 
	\F_0 \ar@{->>}[ll]_\varepsilon&
	& 
	\F_1 \ar@{->>}[dl]_{p_0} \ar[ll]_{d_0}&
	& 
	\F_2 \ar@{->>}[dl]_{p_1} \ar[ll]_{d_1}&
	& 
	\F_3 \ar@{->>}[dl]_{p_2} \ar[ll]_{d_2}& 
	\ar[l] \cdots \\
	&
	&
	& 
	\K_0 \ar@{>->}[ul]_{\iota_0}&
	&
	\K_1 \ar@{>->}[ul]_{\iota_1}&
	&
	\K_2 \ar@{>->}[ul]_{\iota_2}&
}
\]               
We can realize this projective $R_*$ resolution geometrically as a resolution by projective $R$-modules $F_i$ as in \cite{BakerLazarev} and Chapter 4 section 5 of \cite{EKMM}.
Here, define $F_i$ to be the appropriate wedge of $S^n_R$'s so that $\pi_*(F_i)=\F_i$ as an $R_*$-module. 
We define $K_0$ to be the fiber of the map $F_0 \sr{\varepsilon} \lra A$. 

Since the composite 
\[
A \sr{\varepsilon}\lla F_0 \sr{d_0}\lla F_1
\]
is 0 in homotopy it is null-homotopic and we can lift $d_0$ over $\iota_0$ to get the following diagram.
\[
\xymatrix{
&
&
&
F_1 \ar[dl]_{d_0} \ar@{-->}[dr]^{p_0}&
&
&
\\
A&
&
F_0\ar[ll]_{\varepsilon}&
&
K_0\ar[ll]_{\iota_0}&
&
\sus^{-1} A \ar[ll]
}
\]
Now, define $K_1$ to be the fiber of $F_1 \sr{p_0}\lra K_0$. 
That $p_0$ is well defined follows from the projectivity of $F_1$, suppose we had a different map that had the same effect in homotopy, then their difference would be 0 in homotopy and hence null-homotopic.
Next, we repeat this process inductively. 
As $\F_{i-1} \lla \K_{i-1}$ is always an injection, 
\[
\K_{i-1} \lla \F_i \lla \F_{i+1}
\]
is zero since 
\[
\F_{i-1} \lla \F_i \lla \F_{i+1}
\] 
is zero as well. 
Therefore, 
\[
K_{i+1} \lla F_i \lla F_{i+1}
\]
is null-homotopic, and we proceed as above.
This gives us cofiber sequences 
\[
K_{i-1} \sr{p_{i-1}} \lla F_i \sr{\iota_i} \lla K_i
\]
which induce short exact sequences in homotopy.
This implies that the boundary maps $K_{i-1} \sr{\pi_i} \lra \sus K_{i}$ are zero in homotopy.
The factorization
	\[
	\xymatrix{
	\F_i&
	&
	\F_{i+1} \ar[ll]_{d_i} \ar[dl]_{p_i}
	\\
	&
	\K_i \ar[ul]_{\iota_i}
}
\] 
is realized geometrically by
	\[
	\xymatrix{
	F_i&
	&
	F_{i+1} \ar[ll]_{d_i} \ar[dl]_{p_i}
	\\
	&
	K_i \ar[ul]_{\iota_i}
}
\]
since $\pi_*(K_i) \iso \K_i \iso Ker(d_{i-1})\iso Ker(p_{i-1})$. 
$\ $We define $A_i$ to be the cofiber of the map $\sus^{-1}A \sr{\partial_i}\lra \sus^{i} K_i$.
This is the composite of the boundary $\sus^{-1}A \sr{\pi_0} \lra K_0$ (the initial fiber sequence above) with $i$ composites of the $j$-th suspensions of the $i$ boundary maps $\sus^{j-1} K_{j-1} \sr{\pi_j} \lra \sus^j K_j$. 
These give cofiber sequences $\sus^{-1}A \sr{\pi_{i,0}} \lra \sus^i K_i \sr{q_i} \lra A_i$ or rather $\sus^i K_i \sr{q_i} \lra A_i \sr{j_i} \lra A$.
As we already know that the boundary maps $K_{i-1} \sr{\pi_i} \lra \sus K_{i}$ are zero in homotopy, we see that the filtration $\Ab$ is exact.

The following diagram and summary may help in visualizing what is going on.
The top row is the lift of the algebraic resolution of $A_*$ by projective $R_*$ modules to a resolution of $A$ by projective $R$-modules.
Then there is the ``epi-mono'' factorization of this resolution is then the zig zag featuring $F_i$'s and $K_i$'s.
We rotate these cofiber sequences and obtain a tower of $\sus^iK_i$'s from which we obtain our desired filtration.
This filtration $\Ab$ is obtained by taking the levelwise cofiber of the map of towers $\sus^{-1}A \to \sus^{\bu}K_{\bu}$ where the domain is the constant tower.
	
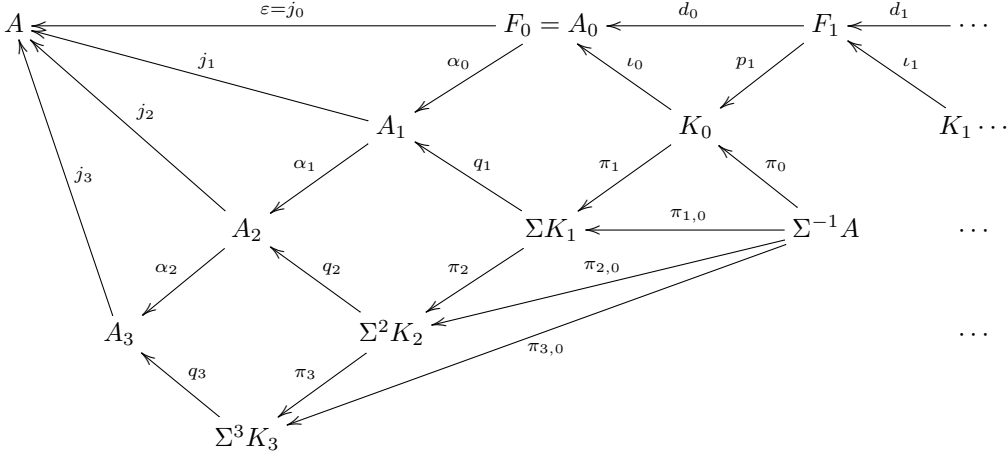
\[
	\xymatrix{
	A&
	&
	&
	&
	F_0=A_0 \ar[llll]_{\varepsilon=j_0} \ar[dl]_{\alpha_0}&
	&
	F_1 \ar[dl]_{p_1} \ar[ll]_{d_0}&
	\cdots \ar[l]_{d_1} \\
	&
	&
	&
	A_1 \ar[ulll]_{j_1} \ar[dl]_{\alpha_1} &
	&
	K_0 \ar[ul]_{\iota_0} \ar[dl]_{\pi_1} &
	&
	K_1 \ar[ul]_{\iota_1}	\cdots \\
	&
	&
	A_2 \ar[uull]_{j_2} \ar[dl]_{\alpha_2}&
	&
	\sus K_1 \ar[dl]_{\pi_2} \ar[ul]_{q_1}&
	&
	\sus^{-1}A \ar[ll]_{\pi_{1,0}} \ar[ul]_{\pi_0} \ar[dlll]_{\pi_{2,0}} \ar[ddllll]^{\pi_{3,0}}&
	\cdots \\
	&
	A_3 \ar[uuul]_{j_3}&
	&
	\sus^2 K_2 \ar[ul]_{q_2} \ar[dl]_{\pi_3}&
	&
	&
	&
	\cdots \\
	&
	&
	\sus^3 K_3 \ar[ul]_{q_3}&
	&
	&
	&
	&
	&
}
\]               
To see that this filtration $\Ab$ is also projective, we compute $\Cof(A_{i-1} \sr{\alpha_{i-1}} \lra A_i)$ using the Verdier braid diagram that analyzes the composition 
\[
\partial_i : \sus^{-1}A \sr{\pi_{i-1,0}}\lra \sus^{i-1} K_{i-1} \sr{\pi_i}\lra \sus^i K_i.
\]

%(note that these are the maps that are used to define the $A_i$)
	\[
	\xymatrix{
	\sus^{-1}A \ar@/^1pc/[rr]_{\partial_i} \ar[dr]^{\pi_{i-1,0}}&
	& 
	\sus^i K_i \ar@/^1pc/[rr]_{\iota_i} \ar[dr]^{q_i}&
	& 
	\sus^i F_i 
	\\
	&
	\sus^{i-1} K_{i-1} \ar[dr]^{q_{i-1}} \ar[ur]_{\pi_i}& 
	&
	A_i \ar[dr]^{j_i} \ar[ur]_{\beta_i}&
	\\
	\sus^{i-1} F_i \ar@/_1pc/[rr]^{q_{i-1}\circ p_i} \ar[ur]_{p_i}& 
	& 
	A_{i-1} \ar[ur]_{\alpha_{i-1}} \ar@/_1pc/[rr]^{j_{i-1}}& 
	& 
	A 
}
\]

Thus the filtration is exact and projective as claimed.

We also want for the maps $A_i \to A_{i+1}$ to be cofibrations.
To see this note that $A_{i+1}$ is the pushout of
\[
\xymatrix{
\sus^i F_{i+1} \ar[rr]^{q_i\circ p_{i+1}} \ar@{^{(}->}[dd]&
&
A_i\\
&
&
\\
\sus^iC(F_{i+1})
}
\]
therefore the map $A_i \to A_{i+1}$ is a cofibration.
\end{proof}

This is the general construction of a filtration.
It is carried out in more detail and with a bit of a different focus in \cite{EKMM}.

\subsection{Monoidal Structures}
\label{subsec:monoidal structures}
In this subsection, we define various multiplicative structures in filtered $R$-modules.
All of our definitions will induce the familiar structure in complexes of $R$-modules.
The goal of these constructions is to provide our spectral sequences with multiplicative structures and power operations whenever the relevant filtration is multiplicative or is an $\Hi$-filtration, respectively.
We first define a monoidal structure on filtered $R$-modules and then proceed to $\Hi$-filtrations.
Our definition of $\Hi$-filtration is as an $\Hi$-object in filtrations as opposed to a filtration by $\Hi$-objects in the underlying category.
The latter leads to a different theory.
For earlier examples of the former, see the work of Bruner \cite{BrunerThesis}, and Hackney \cite{HackneyOpsInfLoopSpace}.
Note that we use a potentially unconventional ``union'' notation for iterated pushouts in $R$-modules.

\begin{dfn}
The \underline{smash product of two filtrations} $\Xb$ and $\Yb$ is denoted by $\Gb(\Xb,\Yb)$.
The $n$th term in the filtration is 
\[
\Gamma_n(\Xb,\Yb):=\bigcup_{i+j=n}X_i\wedge Y_j=X_0 \wedge Y_n \cup_{X_0\wedge Y_{n-1}} X_1 \wedge Y_{n-1} \cup \ldots \cup_{X_{n-1}\wedge Y_0} X_n \wedge Y_0.
\]
We will also denote iterated smash products of $r$-filtrations $\Xb^1,\Xb^2,\ldots,\Xb^r$ by $\Gamma^r(\Xb^1,\Xb^2,\ldots,\Xb^r)_{\bu}$.
Here, the $n$th term in the filtration is
\[
\Gamma^r_n(\Xb^1,\Xb^2,\ldots,\Xb^r):=\bigcup_{\sum_{l=1}^r\alpha_l=n} X^1_{\alpha_1} \wedge X^2_{\alpha_2}\wedge \ldots \wedge X^r_{\alpha_r}.
\]
If all filtrations are the same, we will use the symbol $\Gb^r$ or simply $\Gb$ when $r=2$.
\end{dfn}

\begin{rmk}
We will use relative smash products of $R$-modules.
We will use $\Gb'$ to denote the relative construction where every occurrence of $-\wedge -$ is replaced by $\wedge_R-$ when $R$ is understood from context, otherwise we will use the notation $\Gb^R$.
This will only come up when $R$ is a commutative $S$-algebra. 
\end{rmk}

The above definition allows us to consider monoid objects in the category of filtered modules, in exact analogy with DGA's being monoids in chain complexes.
In fact, the definition of a (strictly) multiplicative filtration is precisely what is necessary to ensure that the associated graded complex of a filtration is a DGA (in spectra) as we will see in Lemma \ref{lemma:smash product of filtrations}.

\begin{dfn}
We say that a filtration $\Xb$ is \underline{multiplicative} if there is a map of filtrations
\[
\xymatrix{
\Gb \ar[rr]_{\mu_{\bu}}&
&
\Xb.
}
\]
In particular, we require maps $\mu_n: \Gamma_n \to X_n$ such that
\[
\xymatrix{
\Gamma_{n-1} \ar[rr] \ar[dd]_{\mu_{n-1}}&
&
\Gamma_n \ar[dd]^{\mu_n}\\
&
&
\\
X_{n-1}\ar[rr]&
&
X_n
}
\]
commutes.
We also require that $\mu_{\bu}$ satisfy the obvious associativity condition up to homotopy.
If the map $\mu_{\bu}$ satisfies the associativity condition on the nose, then we call it strictly multiplicative.
\end{dfn}
The associativity condition is not unreasonable as there is a natural equivalence 
\[
\Gb(\Xb,\Gb(\Yb,Z_{\bu}))\iso\Gb^3(\Xb,\Yb,Z_{\bu})\iso \Gb(\Gb(\Xb,\Yb),Z_{\bu})
\]
which is induced by the associativity of the smash product in the category of $S$-modules.
If $A$ is an $R$-algebra then the constant filtration $c(A)_{\bu}$ is strictly multiplicative.
If a filtration is multiplicative then the (homotopy) colimit of the filtration has a product that is compatible with the filtration.
This follows from the fact that $\colim_{\bu} \Gb(\Xb,\Yb) \we (\colim_{\bu}\Xb)\wedge (\colim_{\bu} \Yb)$. 
The reason for only requiring a weak form of associativity is that we will produce such maps of filtrations by a general construction that produces maps uniquely only up to homotopy and that our applications do not require strictness.
This will be spelled out in the proof of \ref{cor:mult}.

\begin{lemma}
If $\Xb$ and $\Yb$ are two multiplicative filtrations then $\Gb(\Xb,\Yb)$ is multiplicative. 
\end{lemma}

Here we use that the category of $S$-modules is symmetric monoidal and so the related result for $\Gamma^{R}(\Xb,\Yb)$ holds when $R$ is a commutative $S$-algebra.
This result will only be applied in the case that $\Yb$ is a constant filtration of a commutative $S$-algebra.

Notice that there is a natural action of $\pi \subset \Sigma_r$ on the filtration $\Gb^r$ given by permuting factors.
This action, along with the tensoring of the category of $S$-modules over spaces, will be used in the definition of $\Hi$-filtration below.
When necessary, we consider the space $E\Sigma_n$ as a filtered spectrum in $S$-modules via its skeletal filtration.
The model we use for $E\Sigma_n$ as a CW-complex is that which is corresponds to the complex used by May in \cite{MaySteenrod}.

\begin{dfn}
\label{dfn:Hinfty filtration}
We say that a filtration $\Xb$ is \underline{$\Hi$} or has an \underline{$\Hi$-structure} if there are maps of filtrations
\[
\xi_r:\Gamma_{\bu h\Sigma_r}^r:=\Gamma_{\bu}(E\Sigma_{r+}^{(\bu)},\Gb^r(\Xb))_{\Sigma_r} \lra \Xb
\]
that are compatible in the sense of Chapter 1 definition 3.1 of \cite{HRS}.
In particular, we have maps 
\[
\xi_r^n:\Gamma_{n}(E\Sigma_{r+}^{(\bu)},\Gamma_{\bu}^r)_{\Sigma_r}=\bigcup_{i+j=n}E\Sigma^{(i)}_{r+}\wedge_{\Sigma_r} (\bigcup_{\sum_{l=1}^r\alpha_l=j} X_{\alpha_1} \wedge X_{\alpha_2}\wedge \ldots \wedge X_{\alpha_r}) \to X_n.
\]
Or more explicitly, we require 
\[
\xi_r^{n,m}:E\Sigma^{(m)}_{r+}\wedge_{\Sigma_r}(\Gamma_n^r) \to X_{n+m}.
\]
for for every $r,m,$ and $n$.
The $\xi_r^{n,m}$'s must be compatible in the sense that
\[
\xymatrix{
\displaystyle E\Sigma_{r+}^{(n)}\wedge_{\Sigma_r} \Gamma^r_{m-1} \bigcup_{E\Sigma_{r+}^{(n-1)}\wedge_{\Sigma_r} \Gamma^r_{m-1}} E\Sigma_{r+}^{(n-1)} \wedge_{\Sigma_r} \Gamma^r_m \ar[dd]_{\xi_r^{n,m-1}\cup_{\xi_r^{n-1,m-1}} \xi_r^{n-1,m}} \ar[rr]&
&
E\Sigma_{r+}^{(n)} \wedge_{\Sigma_r} \Gamma^r_m \ar[dd]^{\xi_r^{n,m}}\\
&
&
\\
X_{n+m-1} \ar[rr]&
&
X_{n+m}
\\
}
\]
commutes for all $r,n$ and $m$ up to homotopy.
\end{dfn}
We will frequently abbreviate $E\Sigma_{r+}^{(k)}\wedge_{\pi} \Gamma^r_n$ as $\tGamma^r_{k,n}$.
When we restrict to a subgroup $\pi \subset \Sigma_r$ we will use $\tGamma^{\pi}_{k,n}$ to denote the above construction with $\pi$ in place of $\Sigma_r$.
The above definitions are motivated by the following lemmas.
They were chosen so that the associated graded of the filtrations will be exactly the familiar object in chain complexes.
\begin{lemma}
\label{lemma:smash product of filtrations}
Let $\Xb \subset X$ and $\Yb \subset Y$ be bounded below cellular filtrations.
Let the associated graded complexes be denoted by $E^0(\Xb)=F_*$ and $E^0(\Yb)=G_*$.
Then, $\Gb(\Xb,\Yb)$ is a cellular filtration of $X \wedge Y$.
Now suppose that $\Xb$ and $\Yb$ are projective filtrations.
The associated graded complex of $\Gb(\Xb,\Yb)$ is the smash product of the two complexes; that is, 
\[
E^0(\Gb(\Xb,\Yb))=E^0(\Xb)\wedge E^0(\Yb)=F_*\wedge G_*
\]
where the $n$th term in the complex $F_*\wedge G_*$ is given by $\bigvee_{i+j=n}F_i \wedge G_j$.
Also, the associated graded complex of $\Gb'(\Xb,\Yb)$ is the relative smash product of the two complexes 
\[
E^0(\Gb'(\Xb,\Yb))=E^0(\Xb)\wedge_R E^0(\Yb)=F_*\wedge_R G_*
\]
where the $n$th term in the complex $F_*\wedge_R G_*$ is given by $\bigvee_{i+j=n}F_i \wedge_R G_j$.
\end{lemma}
\begin{lemma}
\label{lem:geometric condition}
Let $\Xb\subset X$ be a bounded below cellular projective filtration with associated graded complex $E^0(\Xb)=F_*$.
We then have the following regarding the filtration $\widetilde{\Gamma}_{\bu}^{\pi}(\Xb)$ of $E\pi_+ \wedge_{\pi} X^{\wedge r}$ where 
\[
\widetilde{\Gamma}^{\pi}_m(\Xb):=\Gamma_m(E\pi^{\bu}_+,\Gb^r)=\cup_{k+n=m} \tGamma^{\pi}_{k,n}.
\]
\begin{itemize}
\item $\Gb^r$ and $\widetilde{\Gamma}_{\bu}^{\pi}(\Xb)$ are both bounded below cellular filtrations.
\item The associated graded complex of $\widetilde{\Gamma}^{\pi}_{\bu}(\Xb)$ is given by the smash products of the associated graded complexes of $E\pi^{(\bu)}_+$ and $\Gamma^r_{\bu}$.
In particular,
\[
\frac{\widetilde{\Gamma}^{\pi}_n(\Xb)}{\widetilde{\Gamma}^{\pi}_{n-1}(\Xb)} \we \bigvee_{i+j=n}  \frac{B\pi^{(i)}}{B\pi^{(i-1)}} \wedge (\bigvee_{\sum_{k=1}^r\alpha_k=j}(\bigwedge_{k=1}^r\sus^{\alpha_k}F_{\alpha_k}))
\]
where $B\pi^{(i)}=E\pi^{(i)}/\pi$.
\end{itemize}
\end{lemma}
One should compare the above to Lemma 5.1 of Chapter 4 on page 111 in \cite{HRS}.

As before, if $A$ is a commutative $R$-algebra then the filtration $c(A)_{\bu}$ is a genuinely commutative multiplicative filtration.
We use the term $\Hi$-filtrations as opposed to $\Ei$-filtrations because while we will be able to construct the relevant maps for an $\Ei$-structure we will only be able to show that the relevant diagrams commute up to homotopy.
In fact, since our notion depends on a filtration of the operad there could be multiple different notions that might deserve this name.
This will not surprise those familiar with the algebraic approach to Steenrod and Dyer-Lashof operations of Peter May in \cite{MaySteenrod}.
This inability to construct genuinely commuting diagrams is due to our tool for constructing such structures being the Comparison Theorem, Theorem \ref{thm: Comparison Theorem}, which only constructs a unique map up to homotopy.
This definition is equivalent to the one used by Bruner in \cite{HRS} and \cite{BrunerThesis}.
The apparent difference is due to the filtration happening in ``negative'' degrees in \cite{HRS} and so the skeletal filtration degree of $E\Sigma_n$ decreases the Adams filtration.

The existence of a monoidal structure also implies that the category of filtered modules is tensored over $R$-modules, and therefore any other category that $R$-modules happens to be tensored over.
In particular, this allows us to define a notion of homotopy.

\begin{defn}
\label{defn:homotopy of filtered modules}
 Let $f_{\bu},g_{\bu}:\Ab \lra \Bb$ be two maps of filtered $R$-modules.
 We call $H_{\bu}:I^R_{\bu}\wedge_R \Ab \lra \Bb$ a homotopy from $f_{\bu}$ to $g_{\bu}$ if the following diagram commutes.
 \[
  \xymatrix{
  \Gb'(c(R\wedge {0}_+), \Ab) \we \Ab \ar[drr] \ar@/^1pc/[drrrrr]^{f_{\bu}}&
  &
  &
  &
  &
  \\
  &
  &
  \Gb'(I_{\bu}^R, \Ab) \ar[rrr]^{H_{\bu}} \ar[rrr]&
  &
  &
  \Bb\\
  \Gb'(c(R\wedge {1}_+), \Ab) \we \Ab \ar[urr] \ar@/_1pc/[urrrrr]_{g_{\bu}}&
  &
  &
  &
  &
  }
 \]
Here, $I^R_{\bu}$ is the filtered $R$-module coming from the cellular structure on the standard unit interval.
In filtration $0$ we have that $I^R_0:=R\wedge \{0,1\}_+\we R\vee R$ where the first wedge summand comes from ${0}\subset I$ and the second from ${1} \subset I$.
In filtration $n$ we have that $I^R_{n}:= R \wedge I_+$ and the maps in the filtration are the obvious ones.
\end{defn}

Note that such a homotopy will induce a chain homotopy on homotopy groups of the associated graded complex of a filtration.
This notion of homotopy is also an equivalence relation, the standard proofs that rely on nice properties of the unit interval apply.

\section{Comparison Theorem}
\label{sec:Comparison Theorem}
In this section we state and prove our Comparison Theorem \ref{thm: Comparison Theorem} as well as provide some of its corollaries regarding multiplicative structures.
This result is a filtered version of the Comparison Theorem of homological algebra.
In fact, the Comparison Theorem of homological algebra can be recovered from our result by applying the ``associated graded complex" functor everywhere.
One can also view this result as a ``cellular approximation" type result.

\begin{thm}[Comparison Theorem]
\label{thm: Comparison Theorem}
Suppose that we have a map $f: Y \to A$ of $R$-modules, an exact filtration $A_{\bullet} \subset A$, and a projective and exhaustive filtration $Y_{\bullet} \subset Y$.
Also, suppose that there exist $f_{-1}: Y_{-1} \to A_{-1}$ such that
\[
\xymatrix{
Y_{-1} \ar[rr] \ar[dd]_{f_{-1}}&
&
Y \ar[dd]^f\\
&
&
\\
A_{-1} \ar[rr]&
&
A
}
\]
commutes.
Then there is a map of filtrations $Y_i \sr{f_i} \lra A_i$ such that $\colim f_i\we f$ under the equivalences $\colim Y_i \we Y$ and $\colim A_i \we A$.
Furthermore, the lift $f_{\bu}$ of $f$ is unique up to homotopy of filtered modules, in the sense of Definition \ref{defn:homotopy of filtered modules}.
\end{thm}
It has been pointed out to us that this theorem is reminiscent of Propositions 7.2.1.5 and 7.2.1.8 of Lurie in \cite{LurieHA}.
The first Proposition is the existence of a map of simplicial objects under comparable hypotheses and the second is the uniqueness up to homotopy.
These simplicial objects play the role of resolutions.

\begin{proof}
We first prove the existence of such a lift by induction and then show that uniqueness is a special case of existence.
The base case is handled by the assumption that we have an $f_{-1}$ making the above diagram commute.
In practice, $Y_{-1}$ and $A_{-1}$ will be a point and so having just a map $f: Y \to A$ is sufficient.
Suppose we have already constructed $f_{k}: Y_{k} \to A_{k}$ so that the diagram
\[
	\xymatrix{
	Y_{k-1} \ar[rr]^{s_{k-1}} \ar[dd]_{f_{k-1}} \ar[drrrr]_{y_{k-1}}&
	&
	Y_k \ar[dd]_{f_k} \ar[drr]^{y_k}&
	&
	\\
	&
	&
	&
	&
	A
	\\
	A_{k-1} \ar[rr]_{\alpha_{k-1}} \ar[urrrr]^{j_{k-1}}&
	&
	A_k \ar[urr]&
	&
	\\}
\]  	
commutes for all $-1\leq k<i$.
We now wish to construct $f_i: Y_i \to A_i$ such that
\[
	\xymatrix{
	Y_{i-1} \ar[rr]^{s_{i-1}} \ar[dd]_{f_{i-1}} \ar[drrrr]_{y_{i-1}}&
	&
	Y_i \ar[dd]_{f_i} \ar[drr]^{y_i}&
	&
	\\
	&
	&
	&
	&
	A
	\\
	A_{i-1} \ar[rr]_{\alpha_{i-1}} \ar[urrrr]&
	&
	A_i \ar[urr]&
	&
	\\}
\] 
commutes.
We have the cofiber sequences 
\[
Y_{i-1} \sr{s_{i-1}} \lra Y_i \sr{\delta_i} \lra \sus^i G_i.
\]
with $G_i$ a projective $R$-module as $\Yb$ is a projective filtration.
For the above diagram to commute we must construct 
\[
f_i : Y_i \to A_i
\]
so that 
\begin{center}
$f_i \circ s_{i-1} = \alpha_{i-1} \circ f_{i-1}$ and $j_i \circ f_i = y_i$.
\end{center}
The map 
\[
Y_{i-1} \sr{y_{i-1}} \to A
\]
factors through $Y_i$ so that $y_{i-1}=y_i \circ s_{i-1}$, as indicated by the above diagram.
Therefore the composite $\sus^{i-1} G_i \to Y_{i-1} \sr {y_{i-1}} \to A$ is 0 in homotopy as $\sus^{i-1} G_i$ is the fiber of $s_{i-1}$.
From the following commutative diagram 
\[
	\xymatrix{
	\pi_* \sus^{i-1} G_{i} \ar[rr] \ar@{.>}[dd]&
	&
	\pi_* Y_{i-1} \ar[dd] \ar[rr]^{\pi_*s_{i-1}}&
	&
	\pi_* Y_i \ar[drr]^{\pi_*y_i}&
	&
	\\
	&
	&
	&
	&
	&
	&
	\pi_* A \\
	Ker(\pi_*j_{i-1})=\pi_*\sus^{i-1}K_{i-1} \ar[rr]&
	&
	\pi_* A_{i-1} \ar[rr]_{\pi_* \alpha_{i-1}}&
	&
	\pi_* A_i \ar[urr]_{\pi_*j_i}
}
\]
we can deduce that the composite 
\[
\sus^{i-1}G_i \lra Y_{i-1} \sr{f_{i-1}}\lra A_{i-1} \sr{\alpha_{i-1}}\lra A_i
\]
is null-homotopic.
As the top row of that commutative diagram comes from a cofiber sequence, we know that the image of $\pi_*\sus^{i-1}G_i$ in $\pi_*A$ is $0$ and therefore contained in the kernel of 
\[
\pi_*j_{i-1}:\pi_*A_{i-1} \lra \pi_*A.
\]
The map $\pi_*\alpha_{i-1} : \pi_* A_{i-1} \to \pi_* A_i$ annihilates the kernel of $j_{i-1}$ as the cofiber sequence
\[
\sus^kK_k \lra A_k \lra A
\]
induces a short exact sequence in homotopy for all $k$ by Lemma \ref{lem:partial=0}.
Thus the composite
\[
\sus^{i-1} G_i \to A_{i-1} \to A_i
\]
is 0 in homotopy and is therefore null-homotopic since $G_i$ is projective.

As the above mentioned composite is null-homotopic, we can extend $f_{i-1}$ along $s_{i-1}$ to all of $Y_i$ 
\[
 \xymatrix{
 \sus^{i-1} G_{i} \ar[rr] \ar[ddrr] \ar@/^1pc/[ddrrrr]^>>>>>>>>>>>>0&
	&
	Y_{i-1} \ar[dd]_{f_{i-1}} \ar[rr]^{s_{i-1}}&
	&
	Y_i \ar[dd]^{\widetilde{f}_i}&
	\\
	&
	&
	&
	&
	&
	\\
	&
	&
	A_{i-1} \ar[rr]_{\alpha_{i-1}}&
	&
	A_i }
\]

so that $\widetilde{f}_i \circ s_{i-1} = \alpha_{i-1} \circ f_{i-1}$.
We now must determine if $y_i = j_i \circ \widetilde{f}_i$.
\[
	\xymatrix{
	Y_i \ar[dd]_{\widetilde{f}_i} \ar[dr]^{y_i}&
	\\
	&
	A\\
	A_i \ar[ur]^{j_i}&
}
\]
Consider the map 
\[
\varphi:=j_i \circ \widetilde{f}_{i} - y_i: Y_i \to A.
\]
We have that
\begin{eqnarray*}
\varphi \circ s_{i-1} &=& (j_i \circ \widetilde{f}_i - y_i) \circ s_{i-1} \\
&=& j_i \circ \alpha_{i-1} \circ f_{i-1} - y_{i-1} \\
&=& j_{i-1} \circ f_{i-1} - y_{i-1}\\
&=&0.
\end{eqnarray*}
since by assumption
\begin{eqnarray*}
y_i \circ s_{i-1} &=& y_{i-1}\\
j_{i-1} \circ f_{i-1} &=& y_{i-1}\\
j_i \circ \alpha_{i-1} &=& j_{i-1}.
\end{eqnarray*}
Therefore, $\varphi: Y_i \to A$ induces a map $d : \sus^i G_i \to A$ 
\[
	\xymatrix{
	Y_{i-1} \ar[drr]_{0} \ar[rr]^{s_{i-1}}&
	&
	Y_i \ar[d]_{\varphi} \ar[rr]^{\delta_i}&
	&
	\sus^i G_i \ar@{.>}[lld]^d\\
	&
	&
	A&
	&
}
\]
so that $d \circ \delta_i= \varphi$.
Since $\partial_i$ is 0 in homotopy and $G_i$ is projective, we have that
\[
\partial_i \circ d :  \sus^i G_i \to  \sus^{i+1} K_i
\]
is null-homotopic.
Therefore, there is a lift 
\[
	\xymatrix{
	&
	&
	\sus^i G_i \ar@{.>}[lld]_{d'} \ar[d]^d \ar[rrd]^0&
	&
	\\
	A_i \ar[rr]^{j_i}&
	&
	A \ar[rr]^{\partial_i}&
	&
	\sus^{i+1} K_i
}
\]
$d'$ of $d$, and so we have $j_i \circ d'=d$ and $d \circ \delta_i= \varphi$.
We now define $f_i = \widetilde{f}_i - d' \circ \delta_i$.
All that remains is to check that $f_i \circ s_{i-1} = \alpha_{i-1} \circ f_{i-1}$ and $j_i \circ f_i = y_i$.
\begin{eqnarray*}
 f_i \circ s_{i-1} & = & (\widetilde{f}_i - d' \circ \delta_i) \circ s_{i-1}\\
  & = & \widetilde{f}_i \circ s_{i-1}\\
  & = & \alpha_{i-1} \circ f_{i-1}
\end{eqnarray*}
and
\begin{eqnarray*}
 j_i \circ f_i & = & j_i \circ (\widetilde{f}_i - d' \circ \delta_i)\\
  & = & j_i \circ \widetilde{f}_i - d \circ \delta_i\\
  & = & j_i \circ \widetilde{f}_i - \varphi\\
  & = & y_i
\end{eqnarray*}
as desired.

To prove that any two such lifts are homotopic we will construct a homotopy using the above existence result.
So assume that we are given two different lifts of $f$, call them $f_{\bu}^1$ and $f_{\bu}^2$.
As $f\we \colim f^i_{\bu}$ we have that $f^1 \we f^2$ where $f^i:=\colim f^i_{\bu}$.
Specifically, we have a map 
\[
 H: I_+^R \wedge_R Y \lra A
\]
which is a homotopy from $f^1$ to $f^2$.
We can now lift this homotopy by applying the above argument to the map $H$ and the filtration $\Gb'(I^R_{\bu},\Yb)$, which is still a projective filtration.
The desired homotopy of maps of filtered $R$-modules follows.
\end{proof}
This key lemma will allow us to construct multiplicative and $\Hi$-filtrations.
The uniqueness up to homotopy will be necessary to show that certain diagrams commute up to homotopy.

One application of the above result is the lifting of maps from the category of $R$-modules to the category of filtered $R$-modules.
Given a map of $R$-modules
\[
\varphi: A \lra B
\]
there is a unique homotopy class of lift
\[
\varphi_{\bu}: \Ab \lra \Bb
\]
where the filtrations $\Ab$ and $\Bb$ are those constructed from a projective $R_*$-resolution of $A_*$ and $B_*$ as described in section \ref{subsec:filtration to resolution and back again}.
This fact will provide for the naturality of the following corollaries.

\begin{cor}
\label{cor:mult}
If $A$ is an $R$-algebra, then the filtration of $A$ associated with a projective resolution of $A$ is also multiplicative.
The construction of this multiplicative filtration is natural up to homotopy.
\end{cor}

\begin{proof}
Let $A_{\bullet} \subset A$ be the filtration associated with the projective resolution.
It is projective and exact by Proposition \ref{prop:resolution to filtration}.
There is a natural map of filtrations
\[
\Gamma_{\bu} \sr{r}\lra \Gamma'_{\bu}
\]
where $\Gamma_n := \bigcup_{i+j=n} A_i \wedge A_j\subset A \wedge A$ and $\Gamma'_n := \bigcup_{i+j=n} A_i \wedge_R A_j \subset A \wedge_R A$.
The filtration quotients of $\Gamma'_{\bu}$ are $\bigvee_{i+j=n} \sus^n F_i \wedge_R F_j$ by Lemma \ref{lemma:smash product of filtrations}.
The filtration quotients are projective as each $F_i$ and $F_j$ are projective $R$-modules.
To show that this filtration is multiplicative, we need to construct maps $\mu_n :\Gamma_n \lra A_n$ that are compatible with the filtrations and are restrictions of the product on $A$.
We obtain the desired map as the composition
\[
\Gamma_{\bu} \sr{r}\lra \Gamma'_{\bu} \sr{\mu'_{\bu}}\lra \Ab
\]
where the second map is constructed by applying the Comparison Lemma to the filtration $\Gamma'_{\bu}$ and the map $\mu':A \wedge_R A \to A$.

To show that that the product is homotopy associative we note that both $\mu'_{\bu}\circ(\mu'_{\bu}\wedge_R 1)$ and $\mu'_{\bu}\circ(1 \wedge_R \mu'_{\bu})$
\[
 \Gamma'_{\bu}(\Ab,\Ab,\Ab) \lra \Ab
\]
are solutions to the problem of lifting $\mu'\circ(\mu'\wedge_R 1)=\mu'\circ(1\wedge_R \mu')$
\[
A\wedge_R A  \wedge_R A \lra A
\]
to a filtered map.
By the uniqueness part of the Comparison Theorem we get the desired homotopy.

To see the naturality suppose we are given a map of $R$-algebras
\[
\varphi: A \lra B
\]
that we have lifted to 
\[
\varphi_{\bu}: \Ab \lra \Bb.
\]
We obtain a homotopy commutative diagram of filtered $R$-modules
\[
\xymatrix{
\Gb(\Ab,\Ab) \ar[rr]^{\mu^A_{\bu}} \ar[dd]_{\varphi_{\bu}\wedge \varphi_{\bu}}&
&
\Ab \ar[dd]^{\varphi_{bu}}\\
&
&
\\
\Gb(\Bb,\Bb) \ar[rr]^{\mu^B_{\bu}}&
&
\Bb.}
\]
The commutativity of the diagram up to homotopy follows as argued above, both composites solve the same lifting problem and are therefore homotopic.
\end{proof}

The proof of Theorem \ref{thm: Comparison Theorem} can also be used when we are mapping out of a cofiber sequence with projective fiber.
A particular case of interest is the skeletal filtration on the extended power construction, see  Definition \ref{dfn:Hinfty filtration}.

\begin{cor}
\label{cor:Hinfty structure on filtration}
Using the notation in Definition \ref{dfn:Hinfty filtration}, the filtration $\Ab$ associated with a projective resolution of a commutative $S$-algebra has a $\Hi$-structure.
Explicitly, there are maps $\tGamma^{\pi}_{n,k}=E\pi^{(k)}_+\wedge_{\pi} \Gamma^r_n \sr {\mu_{n,k}} \lra A_{k+n}$ that refine 
\[
\xymatrix{
A^{\wedge r}\ar[rr] \ar[dr]&
&
A\\
&
E\pi_+\wedge_{\pi} A^{\wedge r} \ar[ur]&
}
\]
and fit into the following commutative diagrams:
\[
\xymatrix{
\tGamma^{\pi}_{k,n-1} \bigcup_{\tGamma^{\pi}_{k-1,n-1}} \tGamma^{\pi}_{k-1,n} \ar[dd] \ar[rr]&
&
\tGamma^{\pi}_{k,n} \ar[dd]
&
\\
&
&
\\
A_{n+k-1} \ar[rr]&
&
A_{n+k}\\
}
\]
where $\pi \subseteq \Sigma_r$.
Further, maps of commutative $R$-algebras induce maps of $\Hi$-filtrations.
\end{cor}
\begin{proof}
This is an application of the methods of the proof of Theorem \ref{thm: Comparison Theorem}.
While we do have a group action present we are only considering the analog of the homotopy orbit construction.
It has no group action and we treat it as a bifiltered spectrum.
When we ``totalize'' this bifiltration to have a single filtered module we use Lemma \ref{lem:geometric condition} to understand the associated graded complex.

As in the case of Corollary \ref{cor:mult}, there is an obvious map of filtrations
\[
\tGamma^{\pi}_{\bu} \lra \tGamma'^{\pi}_{\bu}
\]
where 
\[
\tGamma^{\pi}_n=\bigcup_{i+j=n}E\pi^{(i)}_+\wedge_{\pi} \bigcup_{\sum_{l=1}^r\alpha_l=j} X_{\alpha_1} \wedge X_{\alpha_2}\wedge \ldots \wedge X_{\alpha_r}
\]
and
\[
\tGamma'^{\pi}_{\bu} =\bigcup_{i+j=n}E\pi^{(i)}_+\wedge_{\pi} \bigcup_{\sum_{l=1}^r\alpha_l=j} X_{\alpha_1} \wedge_R X_{\alpha_2}\wedge_R \ldots \wedge_R X_{\alpha_r}.
\]
We want to construct the map $\xi_{k,n}$
\[
 \xymatrix{
 \tGamma'^{\pi}_{k-1,n-1}\ar[rr] \ar[dd]_{\xi_{k-1,n-1}}&
 &
 \tGamma'^{\pi}_{k,n-1} \bigcup \tGamma'^{\pi}_{k-1,n}\ar[rr] \ar[dd]_{\xi_{k,n-1}\cup}^{\xi_{k-1,n}}&
 &
 \tGamma'^{\pi}_{k,n} \ar@{.>}[dd]^{\xi_{k,n}}\\
 &
 &
 &
 &
 \\
 A_{k+n-2}\ar[rr]&
 &
 A_{k+n-1}\ar[rr]&
 &
 A_{k+n}
 }
\]
extending $\xi_{k,n-1} \cup \xi_{k-1,n}$ to all of $\tGamma'^{\pi}_{k,n}$.
Theorem \ref{thm: Comparison Theorem} outlines how we do this when the cofiber of the ``inclusion'' is projective and we are mapping into an exact filtration.
Observe that the cofiber of 
\[
\tGamma'^{\pi}_{k,n-1} \bigcup \tGamma'^{\pi}_{k-1,n} \to \tGamma'^{\pi}_{k,n}
\]
is just
\[
\frac{B\pi^{(k)}}{B\pi^{(k-1)}} \wedge \frac{\Gamma'^r_n}{\Gamma'^r_{n-1}}.
\] 
This is a projective module as $B\pi^{(k)}/B\pi^{(k-1)}$ is a wedge of spheres and $\Gamma'^r_n/\Gamma'^r_{n-1}$ is projective by Lemma \ref{lemma:smash product of filtrations}.
We construct the desired map by induction.
The map on $\tGamma'^{\pi}_{0,n}=\Gamma'^{\pi}_n$ is constructed by the previous corollary.
By induction, we assume we have already constructed $\tGamma'^{\pi}_{i,j} \sr {\xi_{i,j}} \lra A_{i+j}$ for $i+j<k$ that restricts to $\mu'_j$ on $\tGamma'^{\pi}_{0,j}$ (the zero skeleton) and filters the structure map 
\[
 E\pi_+\wedge_{\pi} A^r \sr {\xi} \lra A.
 \]
The rest follows from the argument in the proof of Theorem \ref{thm: Comparison Theorem} applied to the cofiber sequence

\[
 \xymatrix{
 \tGamma'^{\pi}_{k-1,n-1}\ar[rr]&
 &
 \tGamma'^{\pi}_{k,n-1} \bigcup \tGamma'^{\pi}_{k-1,n}\ar[rr]&
 &
\tGamma'^{\pi}_{k,n}.
}
\]

To show that the relevant diagrams commute so that we obtain an $\Hi$-filtration we use the same argument as above when we showed that the filtration was multiplicative.
One demonstrates that there are two different solutions to the relevant lifting problem and the homotopy uniqueness of the lift provided by Theorem \ref{thm: Comparison Theorem} provides a homotopy.
This same argument is also used to show that maps of commutative $R$-algebras induce maps of $\Hi$-filtrations.
\end{proof}
This argument is very similar to that used by May in \cite{MaySteenrod} when one needs to construct the relevant homotopy commutative diagrams which we learned from the retelling in \cite{HRS} by Bruner in Chapter 4 section 2.

\begin{cor}
The K\"unneth spectral sequence 
\[
\Tor^{R_*}_p(A_*,B_*)_q \ssra \pi_{p+q}(A\wedge^L_R B)
\]
is multiplicative whenever $R$ is a commutative $S$-algebra, and both $A$ and $B$ are $R$-algebras.
When $A$ and $B$ are also commutative $R$-algebras then the above KSS has a $\Hi$-filtration.
\end{cor}
We postpone the proof until after our construction of the KSS.

\section{The K\"unneth spectral sequence}
\label{sec:KSS}
Here we construct the K\"unneth spectral sequence.
Our construction follows that of \cite{EKMM} although our notation is slightly different.

\begin{rmk}
\label{rmk: BSS}
Under certain flatness hypothesis, the K\"unneth spectral sequence can be made to look remarkably similar to the B\"okstedt spectral sequence.
However, the filtrations giving rise to these spectral sequences are quite different.
The B\"okstedt spectral sequence comes from the skeletal filtration of s simplicial spectrum.
In many cases of interest, the spectral sequence comes from the skeletal filtration of the simplicial commutative $S$-algebra $S^1\tensor A$ and computes $THH(A)$.
In this situation, the filtration is one of commutative $S$-algebras.
Under the above mentioned flatness hypotheses, the $E_2$-page of this spectral sequence can be computed as the Hochschild homology $\HH\HH_*(A_*)$, or rather $\Tor^{A^e_*}_*(A_*,A_*)$ where $A^e:=A\wedge A^{op}$.
In this situation we have both the K\"unneth and B\"okstedt spectral sequences with the same $E_2$-page converging to the same target.
Despite this, the filtrations are in fact different.
In particular, there is the commutation of the Dyer-Lashof operations with the action of the fundamental class of $S^1$, denoted $\sigma$ and referred to as the suspension.
This result doesn't quite make sense in our context.
For example, in $\pi_*(\HF_2\wedge_{ku}\HF_2)$ we have $Q^2(\oo{2})=\oo{v}$, see Proposition \ref{prop: ku computation}.
Even if $\oo{2}=\sigma(2)$ there is not a power operation on $\pi_*ku$ that takes $2$ to $v$.

Even when $R$ is an augmented $\HF_2$-algebra $\HF_2\wedge_R \HF_2$ does not possess a natural $S^1$-action.
In this situation we can identify $\HF_2\wedge_R \HF_2$ with $\HF_2\wedge_R THH^{\HF_2}(R)$.
We can think of the $S^1$-action on $THH^{\HF_2}(R)$ as coming from rotation on the free loop space as $THH^{\HF_2}(R)$ is the functions on the free loop space of $Spec(R)$.
With this in mind, $\HF_2 \wedge_R \HF_2$ is functions on the based loop space where we give $Spec(R)$ the basepoint coming from the augmentation $R \to \HF_2$.
Based loop spaces don't have an $S^1$ action.
\end{rmk}

\subsection{Construction of the K\"unneth spectral sequence}
First we construct the KSS using the techniques for constructing filtrations in \ref{sec:Filtrations and Complexes}.
Then we show that this filtration is multiplicative.

\begin{thm*}[\cite{EKMM}, Chapter 4 section 6]
Let $R$ be an $S$-algebra and $A,B$ be right and left $R$-modules respectively.
Further assume that $(R,S)$ has the homotopy type of a relative CW $S$-module and that $A$ and $B$ are cell $R$-modules (right and left respectively).
Then there is a strongly convergent upper half plane spectral sequence 
\[
E_2^{s,t}= \Tor^{\pi_* R}_s(\pi_*A,\pi_*B)_t \Longrightarrow \pi_{s+t}(A \wedge_R B).
\]
\end{thm*}

The grading $t$ comes from the internal grading as we are working with graded modules over graded algebras.
This result can be found in Chapter 4 section 6 of \cite{EKMM}, but we recall the construction here.

\begin{proof}
We wish to construct a filtration of $A \wedge_R B$ such that the homotopy of the associated graded of the filtration is the complex that computes $\Tor^{\pi_* R}_q(\pi_*A,\pi_*B)$.
Recall that in section \ref{subsec:filtration to resolution and back again} we constructed an exhaustive projective and exact filtration of a given $R$-module $A$ from a projective $R_*$-resolution of $A_*$.
As before, let us denote the projective $R$-module spectra realizing that resolution by $F_i$, and the $i$th of term of the resulting filtration by $A_i$.
Recall that $A_{-1}=*$ and $A_0=F_0$.
Since the filtration is exhaustive (Lemma \ref{lem:partial=0}), we have that $\hocolim_i A_i \we A$ as an $R$-module.
The K\"unneth filtration is given by 
\[
A_0 \wedge_R B \to A_1 \wedge_R B \to \cdots \to A_i \wedge_R B \to A_{i+1} \wedge_R B \to \cdots.
\]
We will abbreviate $A_i \wedge_R B$ as $X_i$.
This filtration is also exhaustive as smashing with a cell $R$-module commutes with taking homotopy colimits; see Chapter 10 section 3 and Chapter 3 sections 2 and 3 of \cite{EKMM}.

Next, we wish to identify the $E_2$ page of the spectral sequence.
The associated graded complex of this filtration is 
\[
\Cof(X_i\to X_{i+1})=\Cof(A_i \wedge_R B\to A_{i+1} \wedge_R B) \we \Cof(A_i\to A_{i+1}) \wedge_R B.
\]
The associated graded complex of the filtration $A_{\bu}$ is a projective $R$-resolution of $A$ by $R$-modules denoted $F_i$.
Therefore the homotopy of $E^0(\Xb)_*$ is given by $\pi_* (F_i \wedge_R B) \iso \pi_* F_i \tensor_{\pi_*R} \pi_* B$, see Proposition 3.9 in Chapter 3 of \cite{EKMM}.
As $\pi_* \sus^i F_i$ is a projective $\pi_* R$-resolution of $\pi_* A$, by construction, the $E_2$ page of the spectral sequence is $\Tor^{\pi_* R}_*(\pi_*A,\pi_*B)_*$.
\end{proof}

It should be noted that there are potentially other constructions of this spectral sequence.
One in particular where a proof of the multiplicative structure was claimed can be found in \cite{BakerLazarev}.
The proof they supply is, unfortunately, incorrect.
See Subsection \ref{subsubsec:previous work} for more details.

\begin{thm}[Multiplicativity of the K\"unneth spectral sequence]
\label{thm: mult of KSS}
Let $R$ be a commutative $S$-algebra, and $A$ and $B$ be $R$-algebras.
The K\"unneth spectral sequence is multiplicative.
Further, the filtration of the KSS has an $\Hi$-structure that filters the commutative $S$-algebra structure of $A\wedge_R B$.
\end{thm}
\begin{proof}
We will use corollaries \ref{cor:mult} and \ref{cor:Hinfty structure on filtration} to show that the K\"unneth filtration is not only multiplicative but has an $\Hi$-structure.
As $A$ is a commutative $R$-algebra and $\Ab$ is exact and projective, there exists a map of filtrations
\[
 \mu_{\bu}: \Gamma(\Ab,\Ab)_{\bu} \lra \Gamma'(\Ab,\Ab)_{\bu} \lra \Ab
\]
by \ref{cor:mult}.
The commutative $R$-algebra structure on $A \wedge_R B$ is the composition
\[
 (A \wedge_R B) \wedge_R (A\wedge_R B) \sr{1\wedge \tau \wedge 1} \lra (A \wedge_R A) \wedge_R (B \wedge_R B) \lra A\wedge_R B.
\]
We can use $\mu_{\bu}$ above to obtain a filtered version of this map.
Using the fact that there is a filtered/levelwise twist map $\tau_{\bu}$ and that $\Gb(\Xb,\Yb\wedge Z)=\Gb(\Xb,\Yb)\wedge Z$, we obtain the 
\[
 \Gb'(\Ab\wedge_R B, \Ab\wedge_R B) \lra \Ab \wedge_R B
\]
as the composition of $\mu_{\bu}$ smashed with the product of $B$ and the following
\[
 \Gb'(\Ab\wedge_R B, \Ab\wedge_R B)=\Gb'(\Ab,B \wedge_R \Ab\wedge_R B) \sr{1\wedge \tau_{\bu} \wedge 1}\lra \Gb'(\Ab, \Ab\wedge_R B\wedge_R B)=\Gb'(\Ab, \Ab)\wedge_R B \wedge_R B.
\]
More explicitly, the maps 
\[
 A_i \wedge A_j \lra A_i \wedge_R A_j \lra A_{i+j}
\]
can be smashed with the product of $B$ to obtain maps
\[
 A_i \wedge_R B \wedge A_j \wedge_R B \lra A_i \wedge_R B \wedge_R A_j \wedge_R B \lra A_i \wedge_R A_j \wedge_R B \wedge_R B \lra A_{i+j} \wedge_R B
\]

That the associated spectral sequence to such a multiplicative filtration is multiplicative is can be found in \cite{DuggerMultSS} by Dugger.
For a slightly different proof see \cite{TilsonThesis}.

Showing that the filtration is an $\Hi$-filtration is similar.
As $B$ is a commutative $R$-algebra it has a natural $\Hi$-structure.
We can then construct
\[
\tGb'^r(\Ab,c(B)) \lra \Ab \wedge_R B
\]
using the $\Hi$-strucutre on $\Ab$ constructed in \ref{cor:Hinfty structure on filtration}, $\Hi$-structure on $B$, and a cellular approximation of the diagonal map on $E\Sigma_r$.
\end{proof}

\begin{rmk}
Relaxing the assumptions on the commutativity of $R$ is an active area of interest that we are pursuing. Of particular interest is the case when $A$ and $B$ are $R$-algebras where $R$ is only an $E_3$-algebra.
See our forthcoming joint work with Katth\"an on $\HH_*(e_n;\F_p)$ for more details.
\end{rmk}

\subsubsection{Previous work}
\label{subsubsec:previous work}
Previous work of Baker-Lazarev has appeared that claims to prove that the KSS is multiplicative.
Their proof relied on an alternate construction of the KSS.
This construction does not yield the KSS as it does not give a filtration.
However, if we supposed it did give rise to a filtration then that filtration, due to the nature of its construction, would have a trivial homotopy colimit whenever our base ring spectrum had homotopy groups of finite homological dimension.
Their filtration is constructed in two steps.
First they realize the resolution $\F_* \to \pi_*A$ used to compute $\Tor^{\pi_*R}(A,B)$ as a resolution of $R$-modules $F_* \to A$.
They then propose to filter $A$ by the sequence of spectra $A_i:=cof(F_i \to F_{i-1})$ for $i>0$ and $A_0:=F_0$.
If $\pi_*R$ has finite homological dimension then the $F_i$ are eventually trivial $R$-modules and so the $A_i$ must be trivial $R$-modules in the same range.
This ensures that $hocolim_i A_i$ can not be weakly equivalent to $A$ whenever $\pi_*A$ has a finite length projective resolution as a $\pi_*R$-module.

\subsection{Algebraic operations}
\label{subsec:algebraic operations}
Since we have shown that the K\"unneth filtration has an $\Hi$-structure we might expect that there is a theory of algebraic operations in $\Tor$ that arise from the structure induced on the associated graded complex of the K\"unneth filtration.
We follow the outline in chapter IV section 2 of \cite{HRS} in constructing an algebraic $\Hi$ structure on the resolution that we will use to compute $\Tor^{R_*}(A_*,B_*)$.
Our construction of such a structure mirrors that of Bruner's in \cite{HRS} at almost every stage, which, in turn, follows \cite{MaySteenrod}.

\subsubsection{The bar construction}
We work with an explicit model for $\Tor^{R_*}(A_*,B_*)$ which is given by $H_*(B^{\mc{F}}(A_*)\tensor_{R_*} B_*)$.
Here, we use $B^{\mc{F}}_{\bu}(A_*)$ to denote the bar construction associated with the free $R_*$-module comonad $\mc{F}$.
It provides us with a preferred free $R_*$-resolution of $A_*$, \cite{Weibel}.
In this subsection, $\bu$ as a subscript will be used to denote simplicial objects as opposed to filtered ones.

We use the (co)monadic bar construction associated with the free-forgetful adjunction as opposed to the two-sided bar construction arising from the tensor product.
This is necessary because tensoring with $R$ does not necessarily create free modules, which is what we need for our construction.
We follow the discussion in \cite{Weibel} sections 8.4.6, 8.6.8, and 8.6.14.

\begin{dfn}
An \underline{augmented simplicial object} $X_{\bu}$ in a category $\mc{C}$ is a simplicial object $X_{\bu}$ in $\mc{C}$ with a morphism $\varepsilon : X_0 \to X_{-1}$ such that 
\[
\varepsilon \partial_0=\varepsilon \partial_1: X_1 \to X_0 \to X_{-1}.
\]
We call such an augmented simplicial object aspherical if $\pi_n X_{\bu}\iso 0$ for $n \neq 0$ and $\pi_0 X_{\bu}\iso X_{-1}$.
\end{dfn}
When we have an augmented simplicial set and an associated simplicial set given by forgetting the augmentation, we will denote the unaugmented simplicial set by $uX_{\bu}$ and the augmented simplicial set $X^+_{\bu}$. 
Such aspherical augmented simplicial sets $X^+_{\bu}$ are referred to as simplicial resolutions of $X_{-1}$.

We will construct an augmented simplicial $R$-module that is a free resolution of $A$.
The adjunction 
\[
F : Sets \rightleftharpoons R-mods : U
\]
between the category of $R$-modules and the category of Sets gives rise to a comonad $\mc{F}:=F\circ U$.
Here, $F$ assigns to any set the free $R$-module on that set and $U$ gives the underlying set of any $R$-module.
The comonad $\mc{F}$ assigns to every $R$-module the free $R$-module generated by its underlying set.
The counit of the adjunction, denoted by $\epsilon$, is the counit of the comonad and the unit of the adjunction, denoted by $\eta$ allows us to define the comultiplication; see \cite{Weibel} section 8.6.
We obtain a simplicial object $B^{\mc{F}}_{\bu}(A)$ whose $n$-simplices are given by $B_n := B^{\mc{F}}_n(A)=\mc{F}^{n+1}(A)$.
The augmentation is given by the counit $\mc{F}(A) \to A$.

\begin{prop}
The bar construction $B^{\mc{F}}_{\bu}(A)^+$ is an aspherical augmented simplicial $R$-module which is level-wise free.
The underlying augmented simplicial set of $B^{\mc{F}}_{\bu}(A)^+$ has an extra degeneracy
\[
\xymatrix{
A=B_{-1} \ar@/^1pc/[r]^{s_0} &
B_0 \ar[l]_{\varepsilon} \ar@/^1pc/[r]^{s_1}&
B_1 \ar[l]_d \ar@/^1pc/[r]^{s_2}& 
B_2 \ar[l]_d&
\cdots
}
\]
such that each $s_i$ is a map of $R$-modules except $s_0: B_{-1} \to B_0$.
The associated chain complex $T(uB^{\mc{F}}_{\bu}(A))_*$ is chain homotopy equivalent to the chain complex $\underline{A}_*$ concentrated in homological degree $0$ with $\underline{A}_0=A$.
\end{prop}
The proof follows from the construction of an extra degeneracy.
One does have to be careful as the extra degeneracy is only a map of sets when $n=0$.
Despite this, all of the necessary identities hold.

We will use the splitting $s$ of $T((B^{\mc{F}}_{\bu}(A))^+)_*$ explicitly which we formalize in the following definition.
\begin{dfn}
\label{dfn:semi-split}
A resolution of an $R$-module $B$
\[
\xymatrix{
B &
G_0 \ar[l]_{\varepsilon}&
G_1 \ar[l]_d& 
G_2 \ar[l]_d&
\cdots
}
\]
is \underline{semi-split} if there exist maps of $R$-modules $s: G_i \to G_{i+1}$ and a map of sets $s_{-1}:B \to G_0$ such that $\varepsilon s_{-1}=1_B$ and $ds+sd=1$
\end{dfn}

\subsubsection{Construction of the operations}
In the following proposition, we focus on the case where $A$ and $B$ are both graded commutative $R$-algebras and the resolutions are given by taking the bar construction as described above.
The product is a $\rho$ equivariant map $f: A^r \lra A$ for any subgroup of $\rho \subset \Sigma_r$.
In this situation, we get algebraic operations in $\Tor^R(A,B)$ by applying $-\tensor_R B$ to the picture below and appealing to May's machinery. 
We will apply the results in this section to the K\"unneth spectral sequence.
There, we will be interested in the case where $R,A$ and $B$ are the homotopy groups of commutative $S$-algebras as discussed in the previous sections.
The following is an analogue of Lemma 2.3 in Chapter IV of \cite{HRS}

\begin{prop}
\label{prop:operations} 
Let $\rho$ be a subgroup of $\Sigma_r$.
Let $\mc{V}_{*}$ be any $\Z\left[ \rho\right]$ free resolution of $\Z$ such that $\mc{V}_0=\Z\left[\rho\right]$ with generator $e_0$.
Let $X$ and $Y$ be $R$-modules.
Let $X \lla F_0 \lla F_1 \lla \cdots$ be a free resolution of $X$ by $R$-modules, $Y \lla G_0 \lla G_1 \cdots$ a resolution of $Y$ by $R$-modules which is semi-split in the sense of definition \ref{dfn:semi-split}.
Let $f:X^r \lra Y$ be a $\rho$-equivariant map of $R$-modules with $\rho$ acting by permuting the factors on $X^r$ and trivially on $Y$.
Let $\rho$ act on $F^r_{*}$ by permuting the factors, trivially on $G_{*}$, and diagonally on $\mc{V}_{*} \tensor F^r_{*}$.
Give $\mc{V}_{*} \tensor F^r_{*}$ the $R$-module structure induced by that on $F^r_{*}$.
Then, there exists a $\rho$ equivariant chain map $\varphi_{i,j}: \mc{V}_i \tensor_{\Z} F^r_j \lra G_{i+j}$ lifting $f$ that is unique up to equivariant chain homotopy.
\end{prop}

\begin{proof}
Our method of proof will be much like that in \cite{HRS}.
We will first construct the lift $\varphi_{*,*}$ on a given $R \left[ \rho \right]$ basis by use of the adjunction between $R \left[ \rho \right]$-modules and $\Z$-modules.
This has the effect of constructing the $R$ map.
The construction of the equivariant chain homotopy will be similar.

Since $F_{*}$ is free over $R$, so is $F^r_{*}$.
Since $\mc{V}_{*}$ is free as a $\Z[\rho]$ module, it follows that $\mc{V}_{*} \tensor F^r_{*}$ is free over $R$ and $\Z \left[ \rho \right]$. This gives us that $\mc{V}_{*} \tensor F^r_{*}$ is $R \left[ \rho \right]$ free.
$F^r_{*}$ augments to $X^r$ via $\varepsilon^r: F^r_0 \lra X^r$. 
The comparison lemma of homological algebra gives us a lift of $f$, 
\[
\widehat{f}:F^r_{*} \lra G_{*}
\] 
which is isomorphic to
\[
\widehat{f}: \langle e_0\rangle  \tensor F^r_{*} \lra G_{*}.
\]
We extend this map equivariantly to obtain 
\[
\varphi_{0,*}: \mc{V}_0  \tensor F^r_{*} \lra G_{*}.
\]

We have now constructed $\varphi_{0,*}$.
Clearly, $\varepsilon_Y\varphi_{0,0}=f \varepsilon_X^r$ by construction, which is the only requirement of $\varphi_{0,0}$ we make.
It remains to construct the rest of the $\varphi_{i,j}$ and verify that they form a chain map.
Let $(i',j')<(i,j)$ if $i'<i$ and $j'\leq j$, or $i'\leq i$ and $j'<j$. 
Suppose we have constructed $\varphi_{i',j'}$ for all $(0,0)<(i',j')<(i,j)$ such that the collection of $\varphi_{i',j'}$'s satisfy the chain map condition in the range they are defined.
We use $\theta$ to denote the natural isomorphism
\[
\theta: \Hom_{\Z[\rho]-R}(F(X),Y) \iso \Hom_R(X,U(Y))
\]
between $R[\rho]$-modules and $R$-modules where $F$ and $U$ are the left and right adjoints respectively.
Then, we define the adjoint $\theta(\varphi_{i,j})$ of the desired map $\varphi_{i,j}$ to be
\[
\theta(\varphi_{i,j})=s(\theta(\varphi_{i-1,j} \circ d \tensor 1) + \theta(\varphi_{i,j-1} \circ 1 \tensor d)).
\]
Since $(0,0)<(i,j)$, we are in the range where $s:G_{i+j-1}\to G_{i+j}$ is a chain null homotopy.
This formula makes more sense in light of the following diagram:

\[
	\xymatrix{
	\mc{V}_i \tensor F^r_j \ar[dd]_{d \tensor 1 \oplus 1 \tensor d} 
	\ar@{.>}[rrr]&
 	&
 	&
 	G_{i+j} \ar[dd]_d
 	\\
 	&
 	&
 	&
 	\\
 	\mc{V}_{i-1} \tensor F^r_j \oplus \mc{V}_i \tensor F^r_{j-1} 
	  \ar[rrr]^(.6){\theta(\varphi_{i-1,j}) +\theta(\varphi_{i,j-1})}
 	&
 	&
 	&
 	G_{i+j-1} \ar@/_1pc/[uu]_s
	}
	\]
All the maps in the above are maps of $R$-modules, not $R[\rho]$-modules; this is the reason for the use of the adjunction.
We have the map defined on a $\Z[\rho]$-basis.
We extend it equivariantly to all of $\mc{V}_i \tensor F^r_j$ using $\theta$, the adjunction isomorphism. 
Now we must verify that $\varphi_{i,j}$, as defined, is indeed a chain map.
Specifically, we wish to show that
\[
d \circ \varphi_{i,j}= \varphi_{i-1,j}\circ d \tensor 1 + \varphi_{i,j-1} \circ 1 \tensor d.
\]
This equation is true if and only if 
\[
\theta (d \circ \varphi_{i,j})= \theta(\varphi_{i-1,j}\circ d \tensor 1 + \varphi_{i,j-1} \circ 1 \tensor d)
\]
By naturality, $\theta(d \circ \varphi)=Ud \circ \theta(\varphi)$. 
Note that $d$ is always a map out of a free $R$-module and that the functor $U$ just forgets that $d$ is an $R$ map, so we will not distinguish between $Ud$ and $d$ as they are the same map element wise.
As we have that
\begin{eqnarray*}
\theta(d \varphi_{i,j}) & = & Ud \theta(\varphi_{i,j})\\
  & = & d \theta(\varphi_{i,j}) \\
  & = & d(s(\theta(\varphi_{i-1,j} \circ d \tensor 1 + \varphi_{i,j-1} \circ 1 \tensor d)) \\
  & = & (1-sd)(\theta(\varphi_{i-1,j} \circ d \tensor 1 + \varphi_{i,j-1} \circ 1 \tensor d))
\end{eqnarray*}
we must show that 
\[
sd(\theta(\varphi_{i-1,j} \circ d \tensor 1) + \theta(\varphi_{i,j-1} \circ 1 \tensor d))
\]
is zero.
Recall that $\varphi_{i-1,j}$ and $\varphi_{i,j-1}$ already commute with differentials below them.
Consider the following expression.

\begin{align*}
& sd(\theta(\varphi_{i-1,j} \circ d \tensor 1 + \varphi_{i,j-1} \circ 1 \tensor d))  =  s\theta(d \varphi_{i-1,j} \circ d \tensor 1 + d \varphi_{i,j-1} \circ 1 \tensor d) \\
 &= \theta((\varphi_{i-2,j}\circ d \tensor 1+\varphi_{i-1,j-1}\circ 1 \tensor d) \circ d \tensor 1 + (\varphi_{i-1,j-1} \circ d \tensor 1 + \varphi_{i,j-2} \circ 1 \tensor d) \circ 1 \tensor d) \\
 &= \theta((0-\varphi_{i-1,j-1} \circ d \tensor 1 \circ 1 \tensor d) + (\varphi_{i-1,j-1} \circ d \tensor 1 \circ 1 \tensor d +0) \\
 & = 0
\end{align*}

We now have chain maps $\varphi_{i,j}: \mc{V}_i \tensor F^r_j \lra G_{i+j}$ of $R_*$-modules defined on a $\Z[\rho]$ basis.
To obtain the desired, map we extend equivariantly. 

This choice of lifts is unique up to equivariant chain homotopy.
Suppose we had $\psi_{i,j}: \mc{V}_i \tensor F^r_j \to  G_{i+j}$, another lift of $f$.
From the data of $\psi,\varphi$ and $s:G_{*} \to G_{*+1}$ we construct a $\rho$-equivariant chain null-homotopy $H_{i,j}: \mc{V}_i \tensor F^r_j \to G_{i+j+1}$.
We define $H_{i,j}$ by induction.
Its adjoint is given by
\[
\theta(H_{i,j}) = s \circ (\varphi_{i,j}-\psi_{i,j}-H_{i-1,j}(d \tensor 1)-H_{i,j-1}(1 \tensor d))
\]
for $i,j$ both nonnegative and $0$ otherwise.
We do not differentiate between $H$ and its adjoint as one is a restriction of the other to a $\rho$-basis and it is notationally clumsy.

We now must verify that $H$ is a chain homotopy between $\varphi$ and $\psi$; that is,
\[
dH_{i,j} + H_{i-1,j}(d\tensor 1) + H_{i,j-1}(1\tensor d)=\varphi_{i,j}-\psi_{i,j}.
\]
We suppose that $H_{i',j'}$ is a chain homotopy for all $(i',j')\leq (i,j)$ as described by the above formula.
We now compute $dH_{i,j}$:
\begin{align*}
&
dH_{i,j}=
ds \circ (\varphi_{i,j}-\psi_{i,j}-H_{i-1,j}\circ(d \tensor 1)-H_{i,j-1}\circ(1 \tensor d))\\
&
=(1-sd) \circ (\varphi_{i,j}-\psi_{i,j}-H_{i-1,j}\circ(d \tensor 1)-H_{i,j-1}\circ(1 \tensor d)).
\end{align*}
Consider
\begin{align*}
&d(\varphi_{i,j}-\psi_{i,j}-H_{i-1,j}\circ(d \tensor 1)-H_{i,j-1}\circ(1 \tensor d))\\
&=d\varphi_{i,j}-d\psi_{i,j}-d(H_{i-1,j}\circ(d \tensor 1)+H_{i,j-1}\circ(1 \tensor d))\\
&=\varphi_{i-1,j}(d\tensor 1) +\varphi_{i,j-1}(1 \tensor d)-\psi_{i-1,j}(d\tensor 1) -\\ 
&\qquad \psi_{i,j-1}(1 \tensor d)-d(H_{i-1,j}\circ(d \tensor 1)+H_{i,j-1}\circ(1 \tensor d))
\end{align*}
We have that
\begin{align*}
&
dH_{i-1,j}(d\tensor 1)=(-H_{i-2,j}(d\tensor 1) - H_{i-1,j-1}(1\tensor d)+\varphi_{i-1,j}-\psi_{i-1,j})\circ(d\tensor 1)\\
&
=-H_{i-1,j-1}(1\tensor d)(d \tensor 1)+(\varphi_{i-1,j}-\psi_{i-1,j})(d\tensor 1)
\end{align*}
as well as
\begin{align*}
&
dH_{i,j-1}(1\tensor d)=(-H_{i-1,j-1}(d\tensor 1) - H_{i,j-2}(1\tensor d)+\varphi_{i,j-1}-\psi_{i,j-1})\circ(1\tensor d)\\
&
=-H_{i-1,j-1}(d\tensor 1)(1 \tensor d)+(\varphi_{i,j-1}-\psi_{i,j-1})(1\tensor d)
\end{align*}
and so 
\[
d(\varphi_{i,j}-\psi_{i,j}-H_{i-1,j}\circ(d \tensor 1)-H_{i,j-1}\circ(1 \tensor d))
\]
simplifies to
\[
-H_{i-1,j-1}(d\tensor 1)(1 \tensor d)-H_{i-1,j-1}(1\tensor d)(d \tensor 1)
\]
which is $0$ as $(d\tensor 1)(1\tensor d)=-(1\tensor d)(d\tensor 1)$.
Now, when $(i,j)=(0,0)$, we have
\[
H_{0,0}=(\varphi_{0,0}-\psi_{0,0}-H_{-1,0}\circ(d \tensor 1)-H_{0,-1}\circ(1 \tensor d))=(\varphi_{0,0}-\psi_{0,0})
\]
as both $H_{-1,0}$ and $H_{0,-1}$ are trivial by construction.
By assumption, we have that $\varepsilon\varphi_{0,0}=f\varepsilon^r=\varepsilon\psi_{0,0}$.
We then extend $H_{i,j}$ to an equivariant chain homotopy as desired.
\end{proof}
The above is enough to construct power operations in $\Tor$.

\begin{cor}
\label{cor:operationsTor}
For $R$ a (graded) commutative algebra, and $A,B$ (graded) commutative $R$-algebras, there exist additive power operations in $\Tor^R_*(A,B)_*$ 
\[
Q_i: \Tor^{R}_s(A,B)_t \lra \Tor^{R}_{i+ps}(A,B)_{pt}
\]
whenever $A$ and $B$ have characteristic $p$.
\end{cor}
\begin{proof}
We apply the above proposition to the situation where $A$ plays both the role of $X$ and $Y$ and the resolutions $F_{*}$ and $G_{*}$ are given by by the associated chain complex of the bar construction $T(uB^{\mc{F}}_{\bu}(A))$ and $\rho=C_p$, the cyclic group of order $p$.
Let $EC_{p*}$ be a resolution of $R$ by free $R\left[C_p\right]$-modules.
We then have the 
\[
\varphi_{*,*}: EC_{p*}\tensor T(uB^{\mc{F}}_{\bu}(A))^p \to T(uB^{\mc{F}}_{\bu}(A))
\]
which becomes
\[
\Phi_{*,*}: EC_{p*}\tensor (T(uB^{\mc{F}}_{\bu}(A)) \tensor_R B)^p \to T(uB^{\mc{F}}_{\bu}(A))\tensor_R B
\]
upon tensoring with $\mu_B: B^p \to B$.
The operation
\[
Q_i: \Tor^{R}_s(A,B)_t \to \Tor^{R}_{i+ps}(A,B)_{pt} 
\]
is given by $Q_i(x):=\Phi_{i,ps}(e_i \tensor x^{\tensor p})$.
(When appropriately re-indexed, these operations satisfy the Cartan formula and Adem relations \cite{MaySteenrod}.)
\end{proof}
While it appears that the definition of these operations depend on the fact that we resolved $A$ as opposed to $B$ this is not the case.
There are quasi-isomorphisms
\[
\widetilde{A}_* \tensor_R B \sr{\widetilde{}}\lla \widetilde{A}_* \tensor_R \widetilde{B}_* \sr{\widetilde{}}\lra A\tensor_R \widetilde{B}_*
\]
where $\widetilde{A}$ is a projective resolution of $A$ and $\widetilde{B}$ is a projective resolution of $B$ as $R$-modules.

Unfortunately, these operations do not turn out to be very interesting.
An application of the below result of Tate will show that these operations will vanish.

\begin{thm}[Tate \cite{Tate}]
\label{thm: Tate}
If $R$ is a Noetherian local ring and $I$ is an ideal of $R$, then there exists a free $R$-resolution of $R/I$ that is a graded commutative DGA.
\end{thm}

Tate proves the above by explicitly constructing a graded commutative DGA.
It can also be extended as follows.
\begin{cor}
For $R$ as above and any finitely generated commutative $R$-algebra $A$, there exists a free $R$-resolution of $A$ that is a graded commutative DGA.
\end{cor}

\begin{proof}
Suppose that $A$ is a finitely generated $R$-algebra, then $A$ is a quotient of $R[x_1, x_2, \ldots,x_n]$ by an ideal.
The result of Tate gives a free $R[x_1, x_2, \ldots,x_n]$-resolution of $A$ that is graded commutative as a DGA.
However, since  $R[x_1, x_2, \ldots,x_n]$ is free over $R$, so is the above resolution and we have the desired result.
\end{proof}

The operations above come from lifting null-homotopies of the difference of ways of ordering $p$-fold products of classes in the free resolution.
If the free resolution is graded commutative, then all of the $p$-fold products are equal and we can take the constant null-homotopy.
This lift will only give trivial operations. 
As the choice of lift is unique up to chain homotopy, the operations are necessarily trivial.

However, the filtration has the requisite structure, so we compute operations in $\pi_* A \wedge_R B$ using the naturality of the $\Hi$-structure and detect them in the KSS
\[
\Tor^{R_*}_*(A_*,B_*)_* \ssra \pi_* (A \wedge_R B).
\]
It might be said that the spectral sequence hides the operations in a lower filtration.
Even when the KSS collapses at $E_2$, there is a difference between operations in $\Tor$ and operations in homotopy as we will see in \ref{sec:comps}.
Any reference to operations will of course be to homotopical operations that only exist on the spectral sequence.
Essentially, the operations on induced by the $\Hi$-structure on the filtration lift one filtration higher than expected and so we see that they are not algebraic.
One could imagine attempting to construct operations on some later page of the spectral sequence.
However, there is no finite page at which all operations will eventually be defined algebraically.
Such a construction would imply that there is a limit to how far the homotopical operations can decrease filtration (beyond what is predicted by the $\Hi$-structure) and examples show that there is no lower bound to this.

\section{Computations}
\label{sec:comps}
In this section we use the above results to compute the action of the Dyer-Lashof algebra on various relative smash products of $\Hk$ where $\mathrm{k}$ is a field of prime order.
Our main use of the results in the preceding section is the naturality of the operations.
Ideally, we would be able to compute the operations algebraically, however we are unable to do this in light of the result of Tate, see Theorem \ref{thm: Tate}.
Thus we are left with computing the operations by using the naturality of the $\Hi$-structure on the filtrations.
These give geometrically defined operations on the level of commutative $\HF_p$-algebras as opposed to algebraic operations on the level of the $E_2$-page of a spectral sequence.
The operations on the homotopy of these commutative $\HF_p$-algebras are those that are constructed by in the sense of May.

The classes related by the operations in this section will be explicitly constructed in next section.
This particular construction gives an interpretation of the operations as a dependence relation between elements in the homotopy groups that is not obvious otherwise.
This is elaborated on further in Section  \ref{sec:applications and interpretations}.
Later, in Subsection \ref{subsec: odd case}, we will address the difference between the prime $2$ and odd primes.
The main difference is in the structure of the dual Steenrod.
There is also a difference between $ku_{(p)}$ and the Adams summand at odd primes.

The computation proceeds in three steps 

\begin{itemize}
 \item Use the Koszul complex to compute the $E_2$-pages of the KSS's converging to $\pi_* \Hk\wedge R \wedge_R \Hk$ and $\pi_* \Hk \wedge_R \Hk$,
 \item Determine the collapse of the spectral sequence as well as the product structure from the multiplicativity of the KSS,
 \item Use step 1 to compute the map of $E_2$-pages and deduce the action of the Dyer-Lashof algebra using this map and the computation of Steinberger \cite{HRS}.
\end{itemize}

To do this, it will help to have some notation in place.
This notation will also be used in section \ref{sec:applications and interpretations} when we interpret these computations.
\begin{dfn}
Given a diagram of commutative $S$-algebras
\[
\xymatrix{
R \ar[rr]^{\psi} \ar[dr]_{\varphi}&
&
R' \ar[dl]\\
&
\Hk&
}
\]
we denote 
\begin{itemize}
 \item the induced map of relative smash products by $\widetilde{\psi}: \Hk \wedge_R \Hk \to \Hk \wedge_{R'} \Hk$,
 \item the induced map of KSSs by $\widehat{\psi}_*: \Tor^{R_*}_s(\kk,\kk)_t \to \Tor^{R'_*}_s(\kk,\kk)_t$,
 \item the induced map of KSSs by $\check{\varphi}_*: \Tor^{R_*}_s(\HH_*(R;\kk),\kk)_t \to \Tor^{R_*}_s(\kk,\kk)_t$.
\end{itemize}
\end{dfn}

We fix a map of commutative $S$-algebras $\varphi: R \to \Hk$.
The map $\varphi$ makes $\Hk$ into a commutative $R$-algebra.
Therefore the action map of $R$ on $\Hk$ coming from said algebra structure is a map of commutative $S$-algebras.
This ensures that the map $\check{\varphi}$ of spectral sequences commutes with all of the extra structure present on the filtration, such as power operations.
To compute the action of the Dyer-Lashof algebra on $\pi_* \Hk \wedge_R \Hk$, consider the map of spectral sequences induced by $\varphi$
\[
\xymatrix{
\Tor^{R_*}_s(\HH_*(R;\kk),\kk)_t \ar@{=>}[rr] \ar[d]_{\check{\varphi}}&
&
\pi_{s+t} (\Hk \wedge R \wedge_R \Hk)\iso \HH_{s+t} (\Hk;\kk) \ar[d]^{\check{\varphi}}\\
\Tor^{R_*}_s(\kk,\kk)_t \ar@{=>}[rr]&
&
\pi_{s+t}(\Hk \wedge_R \Hk).
}
\]

When $\kk$ is a quotient of $R_*$ by a regular sequence, we can use a Koszul complex to compute the $E_2$-page of the spectral sequence.
The spectral sequence is multiplicative and the $E_2$-page is multiplicatively generated by classes on the 1-line of the spectral sequence.
As $d_r=0$ when restricted to the 1-line for every $r \geq 2$ we obtain a collapsing spectral sequence.
In our examples, $\varphi_*:R_* \to \kk$ is a regular quotient, and so both spectral sequences collapse at the $E_2$ page leaving us with
\[
\Tor^{R_*}_s(\HH_*(R;\kk),\kk)_t \iso \HH_{s+t}(\Hk;\kk)
\]
\begin{center} 
and 
\end{center}
\[
\Tor^{R_*}_s(\kk,\kk)_t \iso \pi_{s+t} \Hk \wedge_R \Hk.
\]

We will also need the following computational ingredients.
\begin{thm}[Milnor]
At the prime $2$, the dual Steenrod algebra is $\HH_*(\HF_2;\F_2) \iso \F_2 [\xi_1, \xi_2, \xi_3, \ldots]$.
The conjugation is given by 
\[
\oxi_i:=\chi_*(\xi_i)=\sum\limits_{\alpha \in Part(i)}\prod\limits_{n=1}^{l(\alpha)}\xi^{2^{\sigma(i)}}_{\alpha(n)}
\]
where $l(\alpha)$ is the length of the ordered partition $\alpha$ of $i$.
\end{thm}
The map $\chi: \Hk\wedge \Hk \to \Hk \wedge \Hk$ is the twist map of $S$-modules.

The following result gives us the action of the Dyer-Lashof algebra on $\HH_*(\HF_2;\F_2)$ and is due to Steinberger; see Chapter 3 Thm 2.2 of cite{HRS}.
\begin{thm}[Steinberger] 
In the dual Steenrod algebra, $\HH_*(\HF_2;\F_2) \iso \F_2 [\xi_1, \xi_2, \xi_3, \ldots]$, we have that $Q^{2^i-2}(\xi_1)=\oxi_i$.
The dual Steenrod algebra, $\HH_*(\HF_2;\F_2)$ is generated as an algebra over the Dyer-Lashof algebra by $\xi_1$.
Further, we have that $Q^{2^i}(\oxi_i)=\oxi_{i+1}$ for $i \geq 1$.
\end{thm}

Recall that the Dyer-Lashof algebra acts on the homotopy of any commutative $\HF_2$ algebra.
We will compute the action of $Q^i$ on elements $\oo{x} \in im(\check{\varphi}:\HH_*(\HF_2;\F_2) \to \pi_* \HF_2 \wedge_R \HF_2)$ as follows.
First, we represent both $\oo{x}$ and its lift $\xt\in\HH_*(\HF_2;\F_2)$ as elements in the $E_2$-page of their respective KSSs.
We then use Steinberger's computation to identify the action of $Q^i$ in the spectral sequence converging to $\HH_*(\HF_2;\F_2)$.
Lastly, we push forward $Q^i(\xt)$ along the map of spectral sequences induced by $\varphi$.
We can compute what the map of spectral sequences induced by $\varphi$ does as we will use the same resolution to compute the $E_2$-pages, which we will demonstrate below.

It will be important to determine which class in $\HH_*(\HF_2;\F_2)$ is being detected by given class $\alpha\in \Tor^{R_*}_1(\HF_{2*}R,\F_2)$.
In the cases of interest, the class $\alpha$ will be detecting either $\xi_{i+1}$ or $\oxi_{i+1}$ in $\HH_*(\HF_2;\F_2)$.
The relationship between the $v_i \in \pi_* BP$ and $\oxi_{i+1}, \xi_{i+1} \in \HH_*(\HF_2;\F_2)$ discussed by Ravenel in \cite{GreenBook}, see Chapter 4 section 2 starting on page 114,  implies that $\oo{v}_i$ detects either $\xi_{i+1}$ or $\oxi_{i+1}$.
For our purposes, however, it will not matter if $\oo{v}_i$ detects $\xi_{i+1}$ or $\oxi_{i+1}$.
\begin{lemma}
\label{lemma:decomposables}
The image of $\xi_i$ agrees with $\oxi_i$ modulo decomposables.
\end{lemma}
\begin{proof}
The only indecomposable element in the sum $\tau_*(\xi_i)=\Sigma_{\alpha \in Part(i)}\Pi_{n=1}^{l(\alpha)}\xi^{2^{\sigma(i)}}_{\alpha(n)}$ is $\xi_i$.
Here, $l(\alpha)$ is the length of the ordered partition $\alpha$ of $i$.
While the formula may be complicated, the only indecomposable elements on the right hand side occur when $l(\alpha)=1$, and this forces $\alpha$ to be the indiscrete partition of $i$ that is just $i$ itself.
\end{proof}

\subsection{Relation to algebraic operations}
\label{subsec:relation to alg}
The operations we compute here are Dyer-Lashof operations that act on the homotopy of commutative $\HF_p$-algebras.
Classically, Dyer-Lashof operations refer to operations on the homology of infinite loop spaces or $\Hi$-ring spectra.
However, it is not difficult to see that this is a special case of operations that exist on $\Hi$-$\HF_p$-algebras.

When $E=\HF_p$ and $X=\HF_p\wedge X'$, the operations are referred to as the Dyer-Lashof algebra (or Araki-Kudo algebra when $p=2$).
They are well studied and satisfy a number of desirable properties.
Our familiarity with these operations is due in large part to our understanding of $H_*(D_pS^n;\Z/p)$.
Their most important properties are summarized in the following result taken from chapter III section 1 of \cite{HRS}.
The case $p=2$ is given, while the case $p$ is odd are in parentheses.
\begin{thm}
\label{thm:DLprop}
For any prime $p$ and $s\in\Z$, there exist operations $Q^s$ in the mod $p$ homology of $\Hi$-ring spectra $X$.
They satisfy the following properties.
\begin{itemize}
\item The $Q^s$ are natural homomorphisms $Q^s: H_n(X;\Z/p) \to H_{n+k}(X;\Z/p)$ where $k=s$ ($k=2s(p-1)$).
\item For $x\in H_n(X;\Z/p)$ $Q^s(x)=0$ if $s<n$ ($2s<n$) and $Q^n(x)=x^2$ ($Q^{2n}(x)=x^p$).
\item $Q^s(1)=0$ for $s \neq 0$ where $1\in H_0 (X;\Z/p)$ is the unit.
\item $Q^s(xy)=\displaystyle\sum_{i+j=s} Q^i(x)Q^j(y)$ for $x,y\in H_*(X;\Z/p)$.
\item When $r>2s$ $Q^rQ^s=\displaystyle\sum_i f(r,s,i)Q^{r+s-i}Q^i$ where $f(r,s,i)=\binom{i-s-1}{2i-r}$ $($when $r>ps$ we have $f(r,s,i)=(-1)^{r+i}\binom{(p-1)(i-s)-1}{pi-r})$
\item When $p\neq 2$ and $r \geq ps$, we also have 
\begin{multline*}
Q^r\beta Q^s=\displaystyle\sum_i (-1)^{r+i}\binom{(p-1)(i-s)}{pi-r}\beta Q^{r+s-i}Q^i\\
-\displaystyle\sum_i (-1)^{r+i}\binom{(p-1)(i-s)-1}{pi-r-1}Q^{r+s-i}\beta Q^i
\end{multline*}
\end{itemize}  
\end{thm}
The sums in the last two bullets are not infinite as the binomial coefficients are trivial for all but finitely many $i$.
The Dyer-Lashof algebra also acts on the homotopy of any commutative $\HF_p$-algebra via the construction in \ref{dfn:power op}.

\begin{lemma}
\label{lemma:DLops}
The Dyer-Lashof operations in the homotopy of a commutative $\HF_p$-algebra are the restriction of the Dyer-Lashof operations in the homology of a commutative $\bS$-algebra.
\end{lemma}
\begin{proof}
Let $\eta: S \to \HF_p$ be the unit of the map.
The diagram
\[
\xymatrix{
S^k \ar[r]^{\alpha} \ar[d]_1&
\HF_p \wedge D_r S^n\ar[rr]^{1\wedge D_r(x)} \ar[d]&
&
\HF_p \wedge D_r X \ar[rr]^{1 \wedge \xi_r} \ar[d]&
&
\HF_p \wedge X\ar[r]^{\mu} \ar[d]&
X \ar[d]^1\\
S^k \ar[r]^{\alpha} \ar[d]_1&
\HF_p \wedge D_r S^n\ar[rr]^{1\wedge D_r(\eta_*(x))}\ar[d]&
&
\HF_p \wedge D_r (\HF_p\wedge X) \ar[rr]^{1 \wedge \xi_r} \ar[d]&
&
\HF_p \wedge \HF_p \wedge X\ar[r]^{\mu\circ\mu}\ar[d]&
X\ar[d]^1\\
S^k \ar[r]^{\alpha}&
\HF_p \wedge D_r S^n\ar[rr]^{1\wedge D_r(x)}&
&
\HF_p \wedge D_r X \ar[rr]^{1 \wedge \xi_r}&
&
\HF_p \wedge X\ar[r]^{\mu}&
X
}
\]
commutes.
Here, $\mu': \HF_p \wedge \HF_p \wedge X \to X$
The second row is a homology Dyer-Lashof operation while the top and bottom row are homotopy Dyer-Lashof operations.
The commutativity of this diagram with the identity maps at the right and leftmost columns shows that these two types of operations agree.
Specifically, the Dyer-Lashof operations in the homotopy of a commutative $\HF_p$-algebra are the restriction of the Dyer-Lashof operations in the homology of a commutative $\bS$-algebra.
As 
\[
\xymatrix{
S^k \ar[r]^{\alpha} \ar[d]_1&
\HF_p \wedge D_r S^n\ar[rr]^{1\wedge D_r(x)} \ar[d]&
&
\HF_p \wedge D_r X \ar[rr]^{1 \wedge \xi_r} \ar[d]&
&
\HF_p \wedge X\ar[r]^{\mu} \ar[d]&
X \ar[d]^{\eta}\\
S^k \ar[r]^{\alpha} \ar[d]_1&
\HF_p \wedge D_r S^n\ar[rr]^{1\wedge D_r(\eta_*(x))}\ar[d]&
&
\HF_p \wedge D_r (\HF_p\wedge X) \ar[rr]^{1 \wedge \xi_r} \ar[d]&
&
\HF_p \wedge \HF_p \wedge X\ar[r]^{\mu}\ar[d]&
\HF_p \wedge X\ar[d]^{\mu}\\
S^k \ar[r]^{\alpha}&
\HF_p \wedge D_r S^n\ar[rr]^{1\wedge D_r(x)}&
&
\HF_p \wedge D_r X \ar[rr]^{1 \wedge \xi_r}&
&
\HF_p \wedge X\ar[r]^{\mu}&
X
}
\]
is commutative, they are also related by the Hurewicz map. 
\end{proof}
Further, as the composition 
\[
X \lra \HF_p \wedge X \lra X
\]
is the identity and each map is a map of commutative $\bS$-algebras, $X$ splits off of $\HF_p \wedge X$ as a commutative $\bS$-algebras.
We also have the following identification, which follows from work of McClure in \cite{HRS} as well as a comment of May's in the same volume.
\begin{lemma}
For an $\HF_p$-module $X$, $\pi_* \PP_{\HF_p}X$ is the free allowable algebra over the Dyer-Lashof algebra generated by the graded $\F_p$-vector space $\pi_* X$.
\end{lemma}

We have stated above that we are unable to construct these operations as algebraic operations on the $E_2$-page of the spectral sequence.
What we mean by this is as follows.
The K\"unneth filtration $\Xb$ is an $\Hi$-filtration when considering commtuative $R$-algebras $A$ and $B$ over a commutative ring spectrum $R$,
Thus we have maps
\[
Q_{\alpha}:D^{(k)}_r(X_n) \lra X_{rn+k}
\]
which induce maps on the $E_1$-page and hence all other pages of the spectral sequence.
This map $Q_{\alpha}$ can, and does, factor through lower filtrations.
This factoring will make it impossible to construct all the operations that exist at some finite page of the spectral sequence. 

Let us restrict our attention to the case that $A$ is a commutative $\HF_p$ algebra so that $A\wedge_R B$ is as well.
By the uniqueness of the algebraic operations we constructed in Section \ref{subsec:algebraic operations} and Tate's result, Theorem \ref{thm: Tate}, we know that these operations vanish on the $E_2$-page.
This means that the map above factors as
\[
Q_{\alpha}:D^{(k)}_r(X_n)\lra X_{rn+k-2} \lra X_{rn+k}.
\]
This would naturally lead to operations that could potentially exist on the $E_3$-page of the spectral sequence.
We could then ask about developing such a theory.
However, there is no page at which all of the operations that we know exist homotopically could be defined algebraically.
There is no lower bound $l$ such that all operations $Q_{\alpha}$
\[
D^{(k)}_r(X_n) \sr{Q_{\alpha}}\lra X_{rn+k} \lra X_{rn+k}/X_{rn+k-l}
\]
are nonzero in homotopy.
This can be seen in the following section where eventually $Q^{2^i-2}(\oo{2})$ is detected on the zero line fo the $E_2$-page and hence $Q^{2^i-2}$ factors through the $0$th filtration.
This means that there is no $r$ such that the our homotopical operations, that exist on the level of the filtration, would be detected by an algebraic construction on the $r$th page of the spectral sequence.
Then one might guess that the operations all land in the $0$th filtration, but this is also not the case as we have that $Q^{2^1-2}(\oo{2})=\oo{x}_{2^{i-1}-1}\in \pi_* \HF_2\wedge_{MU} \HF_2$.

\subsection{$ku \lra \HF_2$}
\label{subsec:ku over F_2}
Consider the map $ku \to \HF_2$ whose effect in homotopy is reduction of $ku_*\iso \Z[v]$ modulo $(2,v)$, where $v \in \pi_2 (ku)$ is the Bott class.
This example is the simplest one with nontrivial Dyer-Lashof operations (of which we are aware).
The computations in the following sections will be very similar to the computation here.

\begin{prop}
\label{prop: ku computation}
The KSS $\Tor^{ku_*}_* (\F_2,\F_2) \ssra \pi_*\HF_2 \wedge_{ku} \HF_2$ collapses at the $E_2$ page.
$\pi_* \HF_2 \wedge_{ku} \HF_2$ is an exterior algebra $E_{\F_2}[\oo{2},\oo{v}]$ with $Q^2(\oo{2})=\oo{v}$ where $\vert \oo{2} \vert=1$ and $\vert \oo{v} \vert=3$.
\end{prop}

\begin{proof}
We give this proof in full as the later examples will follow the method closely.
To compute the relevant $\Tor$ group we use the Koszul complex $\Lambda_{ku_*}(2,v)$ associated with the regular sequence $(2,v) \subset ku_*$.
The resolution is

\[
\xymatrix{
ku_*&
&
ku_* \dsum \sus^2 ku_* \ar[ll]_-{
\left( \begin{array}{c}
2\\
v \end{array} \right)}&
&
\sus^2 ku_* \ar[ll]_-{\left( \begin{array}{cc}
v & -2 \end{array} \right)}
}
\]
which becomes
\[
\xymatrix{
\F_2&
\F_2 \dsum \sus^2 \F_2 \ar[l]_-0&
\sus^2 \F_2 \ar[l]_-0
}
\]
after applying $\F_2 \tensor_{ku_*} -$.
We see that $\Tor^{ku_*}_*(\F_2,\F_2)_*$ is an exterior algebra over $\F_2$ generated by classes we call $\oo{2}\in \Tor^{ku_*}_1(\F_2,\F_2)_0$ and $\oo{v}\in \Tor^{ku_*}_1(\F_2,\F_2)_2$.
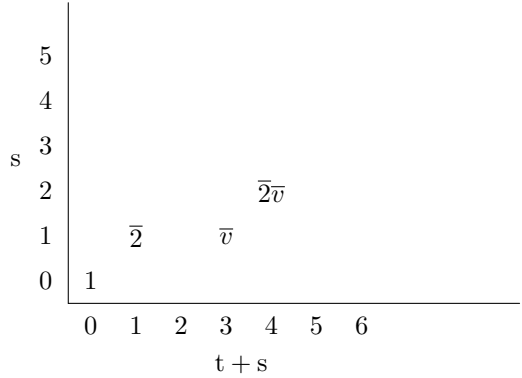
\[
\xy
%labels
(-7,19)*{\mathrm{s}};(23,-8)*{\mathrm{t+s}};
(-3,3)*{0};(-3,9)*{1};(-3,15)*{2};(-3,21)*{3};
(-3,27)*{4};(-3,33)*{5};(3,-3)*{0};(9,-3)*{1};
(15,-3)*{2};(21,-3)*{3};(27,-3)*{4};(33,-3)*{5};
(39,-3)*{6};
%entries
(3,3)*{1};
(9,9)*{\overline{2}};
(21,9)*{\overline{v}};
(27,15)*{\overline{2}\overline{v}};
%axis
\ar@{-}(0,0);(0,40);
\ar@{-}(0,0);(60,0);
%arrows for towers
\endxy
\]
The spectral sequence is multiplicatively generated by $\oo{2}$ and $\oo{v}$ which are on the 1-line of the spectral sequence, and are therefore permanent cycles.

To compute the action of the Dyer-Lashof algebra, we use the induced map of spectral sequences.
We compare the above spectral sequence to
\[
\Tor^{ku_*}_*(\HH_*(ku;\F_2), \F_2) \ssra \pi_* (\HF_2 \wedge ku \wedge_{ku} \HF_2) \iso \HH_*(\HF_2;\F_2).
\]
By Steinberger's Theorem, we know the action of the Dyer-Lashof algebra on $\HH_*(\HF_2;\F_2)$, the target of the spectral sequence.
As $\HH_*(ku;\F_2) \iso \F_2[\oxi_1^2, \oxi_2^2, \oxi_3, \oxi_4, \ldots]$ and both $2$ and $v$ act trivially on $\HH_*(ku;\F_2)$ we can compute $\Tor$ using the same Koszul complex as above.
This complex becomes
\[
\xymatrix{
\HH_*(ku;\F_2) &
&
\HH_*(ku;\F_2) \dsum \sus^2 \HH_*(ku;\F_2) \ar[ll]_-0&
&
\sus^2 \HH_*(ku;\F_2) \ar[ll]_-0
}
\]
after applying $\HH_*(ku;\F_2) \tensor_{ku_*} -$ to $\Lambda_{ku_*}(2,v)$.
Therefore we have that 
\[
\Tor^{ku_*}_*(\HH_*(ku;\F_2),\F_2)_* \iso \HH_*(ku;\F_2) \tensor E_{\F_2}(\oo{2},\oo{v})
\]
with $\oo{2}$ and $\oo{v}$ as above.
This spectral sequence collapses for the same reasons as before.
We know that $\oo{2}$ detects $\oxi_1 \in \HH_1(\HF_2;\F_2)$ as there are no other nonzero classes in those degrees and so $\oo{2}\oxi_1^2$ detects $\oxi_1^3$.
The class $\oo{v}$ detects either $\xi_2$ or $\oxi_2=\xi_2+\xi_1^3$ in $\HH_3(\HF_2;\F_2)$.
It will not matter which, as they agree modulo $\ker(\check{\varphi})$ by \ref{lemma:decomposables}.
See the computation of $\HH_*(ku;\F_2)$ in Chapter 3 of \cite{GreenBook} for the relationship between $v$ and $\oxi_2$ , specifically Theorem 3.1.16.
From Steinberger's computation, we know that there is an operation $Q^2(\oxi_1)=\oxi_2$, this lifts to $Q^2(\oo{2})=\oo{v}+\epsilon\oo{2}\xi_1^2$ where $\epsilon\in\{0,1\}$ in the $E_2$ page of the spectral sequence.

The map of spectral sequences induced by $\varphi: ku \to \HF_2$ is obtained from $\check{\varphi}: \HF_2 \wedge ku \to \HF_2$ which induces quotient map $\HH_*(ku;\F_2) \to \F_2$.
We then obtain the map of Koszul complexes
\[
\xymatrix{
\HH_*(ku;\F_2) \ar[dd] &
&
\HH_*(ku;\F_2) \dsum \sus^2 \HH_*(ku;\F_2) \ar[ll] \ar[dd]&
&
\sus^2 \HH_*(ku;\F_2) \ar[ll] \ar[dd]\\
&
&
&
&
\\
\F_2&
&
\F_2 \dsum \sus^2 \F_2 \ar[ll]&
&
\sus^2 \F_2 \ar[ll]
}
\]
This induces the map of $E_2$-pages $\HH_*(ku;\F_2) \tensor E_{\F_2}(\oo{2},\oo{v}) \to \F_2 \tensor E_{\F_2}(\oo{2},\oo{v})$.
This map preserves the action of the Dyer-Lashof algebra since the construction of the $\Hi$-structure of the filtration in the KSS is natural.
Regardless of the value of $\epsilon$, $\check{\varphi}(\oo{2}\xi_1^2)=0$.
We then have that $\pi_* \HF_2 \wedge_{ku} \HF_2 \iso E_{\F_2}(\oo{2},\oo{v})$ with $Q^2(\oo{2})=\oo{v}$ as claimed.
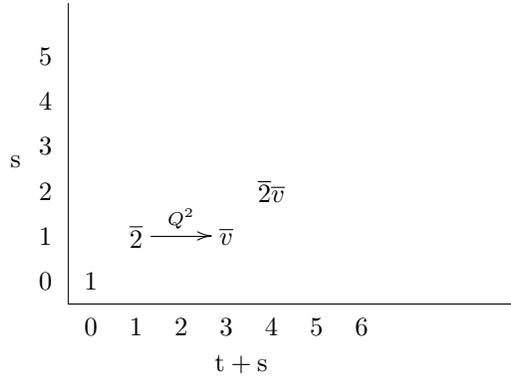
\[
\xy
%labels
(-7,19)*{\mathrm{s}};(23,-8)*{\mathrm{t+s}};
(-3,3)*{0};(-3,9)*{1};(-3,15)*{2};(-3,21)*{3};
(-3,27)*{4};(-3,33)*{5};(3,-3)*{0};(9,-3)*{1};
(15,-3)*{2};(21,-3)*{3};(27,-3)*{4};(33,-3)*{5};
(39,-3)*{6};
%entries
(3,3)*{1};
(9,9)*{\overline{2}};
(21,9)*{\overline{v}};
(27,15)*{\overline{2}\overline{v}};
%axis
\ar@{-}(0,0);(0,40);
\ar@{-}(0,0);(60,0);
%arrows for operations
\ar@{<-}_{Q^2}(11,9);(19,9);
\endxy
\]
\end{proof}

\subsection{$\BP \to \HF_2$}
It was recently shown by Lawson and Naumann in \cite{LawsonNaumann} that the $S$-module $\BP$ can be modeled as a commutative $S$-algebra.
Given this, it is easy to construct a map of commutative $S$-algebras $\BP \to \HF_2$ whose effect in homotopy is reduction of $\BP_*\iso \Z_{(2)}[v_1,v_2]$ modulo $(2,v_1,v_2)$, where $v_1 \in \pi_2 \BP$ and $v_2 \in \pi_6 \BP$.
Using this map, we proceed as in the previous section.
\begin{prop}
The KSS $\Tor^{\BP_*}_* (\F_2,\F_2) \ssra \pi_*\HF_2 \wedge_{\BP} \HF_2$ collapses at the $E_2$ page.
$\pi_* \HF_2 \wedge_{\BP} \HF_2$ is an exterior algebra $E_{\F_2}[\oo{2},\oo{v}_1,\oo{v}_2]$ with $Q^2(\oo{2})=\oo{v}_1$, $Q^6(\oo{2})=\oo{v}_2$, $Q^4(\oo{v}_1)=\oo{v}_2$, and $Q^6(\oo{2}\oo{v}_1)=\oo{v}_1\oo{v}_2$ where $\vert \oo{2} \vert=1$, $\vert \oo{v}_1 \vert=3$ and $\vert \oo{v}_2 \vert=7$.
\end{prop}
\begin{proof}
To compute the relevant $\Tor$ group, we use the Koszul complex associated with the regular sequence $(2,v_1,v_2) \subset \BP_*$.
The resulting complex that computes $\Tor^{\BP_*}_* (\F_2,\F_2)$ is
\[
\F_2 \stackrel{0}{\longleftarrow}
\F_2 \dsum \sus^2 \F_2 \dsum \sus^6 \F_2 
\stackrel{0}{\longleftarrow}
\sus^2 \F_2 \dsum \sus^6 \F_2 \dsum \sus^8 \F_2 
\stackrel{0}{\longleftarrow}
\sus^8 \F_2
\]
We see that $\Tor^{\BP_*}_*(\F_2,\F_2)_*$ is an exterior algebra over $\F_2$ generated by classes $\oo{2}\in \Tor^{\BP_*}_1(\F_2,\F_2)_0$, $\oo{v}_1\in \Tor^{\BP_*}_1(\F_2,\F_2)_2$ and $\oo{v}_2\in \Tor^{\BP_*}_1(\F_2,\F_2)_6$.
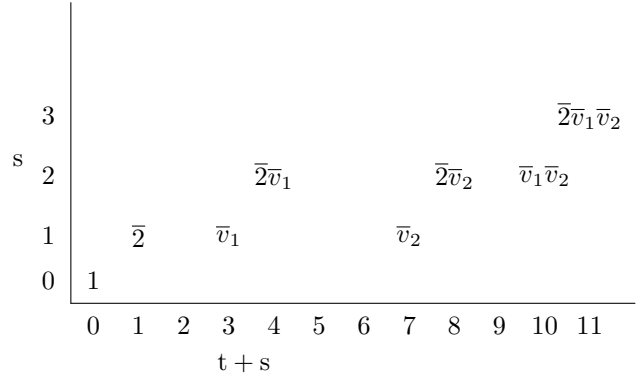
\[
\xy
%labels
(-7,19)*{\mathrm{s}};(23,-8)*{\mathrm{t+s}};
(-3,3)*{0};(-3,9)*{1};(-3,17)*{2};(-3,25)*{3};
(3,-3)*{0};(9,-3)*{1};
(15,-3)*{2};
(21,-3)*{3};(27,-3)*{4};(33,-3)*{5};
(39,-3)*{6};
(45,-3)*{7};(51,-3)*{8};(57,-3)*{9};
(63,-3)*{10};(69,-3)*{11};%(77,-3)*{12};
%entries
(3,3)*{1};
(9,9)*{\overline{2}};
(21,9)*{\overline{v}_1};
(45,9)*{\oo{v}_2};
(27,17)*{\oo{2}\oo{v}_1};
(51,17)*{\oo{2}\oo{v}_2};
(63,17)*{\oo{v}_1\oo{v}_2};
(69,25)*{\oo{2}\oo{v}_1\oo{v}_2};
%axis
\ar@{-}(0,0);(0,40);
\ar@{-}(0,0);(75,0);
%arrows for towers
%\ar@{<-}_{Q^2}(11,9);(25,9);
%\ar@{<-}_{Q^4}(29,9);(42,9);
%\ar@/^1pc/@{<-}^{Q^6}(11,9);(42,8);
\endxy
\]
For exactly the same reasons as in the case of $ku$, this spectral sequence collapses at the $E_2$-page.

We compute the action of the Dyer-Lashof algebra just as in the case of $ku$, by computing the induced map of spectral sequences.
Consider KSS 
\[
\Tor^{\BP_*}_*(\HH_*(\BP;\F_2), \F_2) \ssra \pi_* (\HF_2 \wedge \BP \wedge_{\BP} \HF_2) \iso \HH_*(\HF_2;\F_2).
\]
By Steinberger's Theorem, we know the action of the Dyer-Lashof algebra on $\HH_*(\HF_2;\F_2)$, the target of the spectral sequence.
As $\HH_*(\BP;\F_2) \iso \F_2[\oxi_1^2, \oxi_2^2, \oxi_3^2, \oxi_4, \oxi_5, \ldots]$ with $2$, $v_1$, and $v_2$ all acting trivially, we compute $\Tor$ using same Koszul complex as above.
After applying $\HH_*(\BP;\F_2) \tensor_{\BP_*} -$, it becomes
\[
\HH_*(\BP;\F_2) \stackrel{0}{\longleftarrow}
\begin{array}{c}
\HH_*(\BP;\F_2)_*\\
\dsum\\
\sus^2 \HH_*(\BP;\F_2)_*\\
\dsum\\
\sus^6 \HH_*(\BP;\F_2) 
\end{array} 
\stackrel{0}{\longleftarrow}
\begin{array}{c}
\sus^2 \HH_*(\BP;\F_2)\\
\dsum\\
\sus^6 \HH_*(\BP;\F_2)\\
\dsum\\
\sus^8 \HH_*(\BP;\F_2) \end{array} 
\stackrel{0}{\longleftarrow}
\sus^8 \HH_*(\BP;\F_2). 
\]
We see that 
\[
\Tor^{\BP_*}_*(\HH_*(\BP;\F_2),\F_2)_* \iso \HH_*(\BP;\F_2) \tensor E_{\F_2}(\oo{2},\oo{v}_1,\oo{v}_2)
\]
with degrees as described above.
This spectral sequence also collapses.
As before, $\oo{2}$ detects $\oxi_1 \in \HH_1(\HF_2;\F_2)$, and so $\oo{2}\oxi_1^2$ detects $\oxi_1^3$.
We also have that $\oo{v}_1$ detects either $\xi_2$ or $\oxi_2 =\oxi_2+\oxi_1^3$ in $\HH_3(\HF_2;\F_2)$.
Similarly, $\oo{v}_2$ detects either $\xi_3$ or $\oxi_3$ in $\HH_7(\HF_2;\F_2)$.
Details relating $v_i$ and $\xi_{i+1}$ can be found in Ravenel's \cite{GreenBook} where Ravenel discusses the computation of $\HH^*(k(n);\F_2)$, see Chapter 4 section 2 starting on page 114.
Also see the computation in section 5.2 page 1246 of \cite{AngeltveitRognes}.
From Steinberger's computation we know that $Q^2(\oxi_1)=\oxi_2$, $Q^6(\oxi_1)=\oxi_3$, and $Q^4(\oxi_2)=\oxi_3$ 

The map of spectral sequences induced by $\varphi$ is obtained from $\HF_2 \wedge \BP \to \HF_2$ which induces the reduction map $\HH_*(\BP;\F_2) \to \F_2$.
This induces the map of complexes
\[
\xymatrix{
\HH_*(\BP;\F_2) \ar[dd]&
{\begin{array}{c}
\HH_*(\BP;\F_2)_*\\
\dsum\\
\sus^2 \HH_*(\BP;\F_2)_*\\
\dsum\\
\sus^6 \HH_*(\BP;\F_2) 
\end{array} 
} \ar[l] \ar[dd]&
{\begin{array}{c}
\sus^2 \HH_*(\BP;\F_2)_*\\
\dsum\\
\sus^6 \HH_*(\BP;\F_2)_*\\
\dsum\\
\sus^8 \HH_*(\BP;\F_2) 
\end{array} 
} \ar[l] \ar[dd]&
\sus^8 \HH_*(\BP;\F_2) \ar[l] \ar[dd]\\
&
&
&
\\
\F_2&
\F_2 \dsum \sus^2 \F_2 \dsum \sus^6 \F_2 \ar[l]&
\sus^2 \F_2 \dsum \sus^6 \F_2 \dsum \sus^8 \F_2 \ar[l]&
\sus^8 \F_2. \ar[l]
}
\]
This map of complexes induces the map of $E_2$-pages $\HH_*(\BP;\F_2) \tensor E_{\F_2}(\oo{2},\oo{v}_1, \oo{v}_2) \to \F_2 \tensor E_{\F_2}(\oo{2},\oo{v}_1, \oo{v}_2)$.

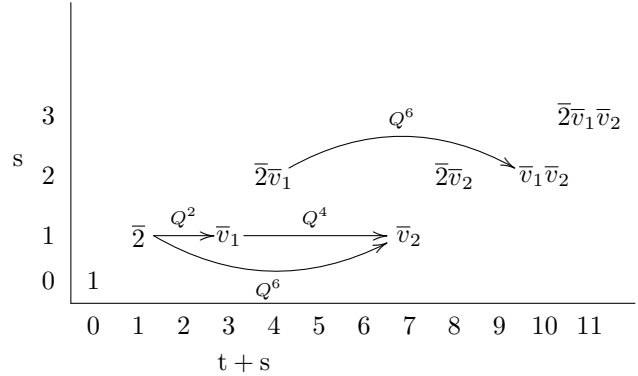
\[
\xy
%labels
(-7,19)*{\mathrm{s}};(23,-8)*{\mathrm{t+s}};
(-3,3)*{0};(-3,9)*{1};(-3,17)*{2};(-3,25)*{3};
(3,-3)*{0};(9,-3)*{1};
(15,-3)*{2};
(21,-3)*{3};(27,-3)*{4};(33,-3)*{5};
(39,-3)*{6};
(45,-3)*{7};(51,-3)*{8};(57,-3)*{9};
(63,-3)*{10};(69,-3)*{11};%(77,-3)*{12};
%entries
(3,3)*{1};
(9,9)*{\oo{2}};
(21,9)*{\oo{v}_1};
(45,9)*{\oo{v}_2};
(27,17)*{\oo{2}\oo{v}_1};
(51,17)*{\oo{2}\oo{v}_2};
(63,17)*{\oo{v}_1\oo{v}_2};
(69,25)*{\oo{2}\oo{v}_1\oo{v}_2};
%axis
\ar@{-}(0,0);(0,40);
\ar@{-}(0,0);(75,0);
%arrows for towers
\ar@{<-}_{Q^2}(11,9);(19,9);
\ar@{<-}_{Q^4}(23,9);(42,9);
\ar@/^1pc/@{<-}^{Q^6}(11,9);(42,8);
\ar@/^1pc/@{->}^{Q^6}(29,18);(59,18)
\endxy
\]

This map preserves the action of the Dyer-Lashof algebra and so we have that $\pi_* \HF_2 \wedge_{\BP} \HF_2 \iso E_{\F_2}(\oo{2},\oo{v}_1,\oo{v}_2)$ with $Q^2(\oo{2})=\oo{v}_1$, $Q^6(\oo{2})=\oo{v}_2$, and $Q^4(\oo{v}_1)=\oo{v}_2$ as desired.
The operation $Q^6(\oo{2}\oo{v}_1)=\oo{v}_1\oo{v}_2$ follows from the Cartan formula.
\end{proof}

\subsection{$MU \to \HF_2$}
In this section we compute some of the action of the Dyer-Lashof algebra on $\pi_*\HF_2\wedge_{MU} \HF_2$.
As before, we will need to know the Hurewicz image for $MU$; for this see section 2 of \cite{BakerRichterAdams} or Chapter 3 of \cite{GreenBook}.
Recall that $\pi_*MU\iso \Z[x_1,x_2,...]$ where $x_i \in \pi_{2i} MU$ and $H_* (MU; \Z) \iso \Z[b_1, b_2, ...]$ where $b_i \in H_{2i} (MU;\Z)$.
Our choice of generators for $\pi_* MU$ is such that $\rho(x_{2^i-1})=v_i$, where 
\[
 \rho: MU \to MU_{(2)}
\]
is the localization map and we think of the $v_i$ as elements of $\pi_*MU_{(2)}$ via a preferred retraction of $MU_{(2)}$ onto $BP$.
This choice of generators for $MU_*$ will determine our choice of generators for $\HH_*(MU;\F_2)$, but this choice can be made so that our Hurewicz map behaves as desired; see \cite{BakerRichterAdams} section 2. 
The Hurewicz map $h_{H\Z}: \pi_* MU \to H_*(MU; \Z)$ modulo decomposables is given by
\[
h(x_i) = \left\{
\begin{array}{ll}
pb_i & i=p^k-1 \\
b_i  & else. \\
\end{array}
\right.
\]
When the Hurewicz map 
\[
h_E: \pi_*MU \to E_* MU
\] is of this form, it implies that 
\[
\Tor^{MU_*}(E_*MU,E_*) \iso \Tor^{BP_*}(E_*BP,E_*).
\]
This is the case for $\HF_2$ and other spectra associated with complex $K$-theory; see Proposition 3.1 of \cite{BakerRichterAdams}.
All that is left is to compute the map $\Tor^{MU_*}(\HH_*(MU;\F_2),\F_2) \to \Tor^{MU_*}(\F_2,\F_2)$.

\begin{prop}
The KSS 
\[
\Tor^{MU_*}_* (\F_2,\F_2) \ssra \pi_*\HF_2 \wedge_{MU} \HF_2
\]
collapses at the $E_2$ page.
$\pi_* \HF_2 \wedge_{MU} \HF_2$ is an exterior algebra $E_{\F_2}[\oo{2},\oo{x}_1,\oo{x}_2, \ldots]$ with $Q^{2^i-2}(\oo{2})=\oo{x}_{2^{i-1}-1}$ and $Q^{2^i}(\oo{x}_{2^{i-1}-1})=\oo{x}_{2^i-1}$ where $\vert \oo{2} \vert=1$ and $\vert \oo{x}_n \vert=2n+1$.
\end{prop}

\begin{proof}
We first compute $\Tor^{MU_*}_*(\F_2,\F_2)_*$ by realizing $\F_2$ as the regular quotient of $MU_*$ by the ideal $(2, x_1, x_2, \ldots)$.
Let us denote this Koszul complex  by $\Lambda:=\Lambda_{MU_*}(\oo{2},\oo{x}_1,\oo{x}_2,\ldots)$.
Since the sequence $(2, x_1, x_2, \ldots)$ is regular, we have that $\HH_*(\Lambda)=\HH_0(\Lambda)=\F_2$, and so
\[
\Tor^{MU_*}_*(\F_2,\F_2):=\HH_*(\F_2\tensor_{MU_*} \Lambda,d).
\]
The $MU_*$-module structure on $\F_2$ is the trivial one, so we have that $d(\oo{x}_i)$ acts trivially on $\F_2$.
This implies that 
\[
\Tor^{MU_*}_*(\F_2,\F_2)=E_{\F_2}[\oo{2},\oo{x}_1,\oo{x}_2, \ldots].
\]
As before, the $E_2$-page of this spectral sequence is multiplicatively generated on the 1-line.
Therefore the spectral sequence collapses, and we have that 
\[
\pi_* \HF_2 \wedge_{MU} \HF_2 \iso E_{\F_2}[\oo{2},\oo{x}_1,\oo{x}_2, \ldots].
\]

Next, we compare this spectral sequence with 
\[
\Tor^{MU_*}_*(\HH_*(MU;\F_2),\F_2)_* \ssra \HH_*(\HF_2;\F_2)
\]
in order to compute the action of the Dyer-Lashof algebra.
The Hurewicz map determines the $MU_*$-module structure of $\HH_*(MU;\F_2)$, which is nontrivial.
In using the complex $\Lambda$, we obtain the sequence of isomorphisms 
\[
\Tor^{MU_*}_* (\HH_*(MU;\F_2),\F_2)_* \iso \HH_*(\HH_*(MU;\F_2) \tensor_{MU_*} \Lambda ,d)\iso \Tor^{BP_*}_*(\HH_*(BP;\F_2),\F_2).
\]
This follows as the only cycles come from the classes in $MU_*$ that survive the map $MU \to BP$.
Again, we see that the spectral sequence collapses at the $E_2$-page.

To compute the action of the Dyer-Lashof algebra, we proceed as above, utilizing the fact that the element
\[
\oo{x}_{2^i-1}=\oo{v}_i\in\Tor^{BP_*}_1(\HH_*(BP;\F_2),\F_2)\iso (\F_2[\xi_1^2,\xi_2^2,\xi_3^2,\ldots]\tensor E_{\F_2}[\oo{2},\oo{x}_1,\oo{x}_3,\ldots,\oo{x}_{2^i-1},\ldots])
\]
detects $\oxi_{i+1}$ or $\xi_{i+1}$ in $\HH_{2^{i+1}-1}(\HF_2; \F_2)$.
On the $E_2$-page, this map is 
\[
\HH_*(BP;\F_2) \tensor E_{\F_2}[\oo{2},\oo{x}_1,\oo{x}_3,\ldots,\oo{x}_{2^i-1},\ldots] \to \F_2 \tensor E_{\F_2}[\oo{2},\oo{x}_1,\oo{x}_2,\oo{x}_3,\ldots].
\]
This map preserves the action of the Dyer-Lashof algebra, and this gives us that $Q^{2^i-2}(\oo{2})=\oo{x}_{2^{i-1}-1}$ and $Q^{2^i}(\oo{x}_{2^{i-1}-1})=\oo{x}_{2^i-1}$ as desired.
\end{proof}

This computation has the following corollary.

\begin{cor}
\label{cor:strickland analog}
Let $I$ be an ideal of $MU_*$ generated by a regular sequence. 
If $I$ contains a finite nonzero number of the $x_{2^i-1}$, then the quotient map $MU \to MU/I$ cannot be realized as a map of commutative $S$-algebras.
\end{cor}

\begin{proof}
Let $I$ be as above and let us use $E$ to denote $MU/I$.
Suppose that the quotient map 
\[
\varphi:MU \to E
\]
is a map of commutative $S$-algebras.
Consider the induced map of the spectral sequences converging to 
\[
\widetilde{\varphi}: \HF_2 \wedge_{MU} \HF_2 \to \HF_2 \wedge_E \HF_2.
\]
This map must preserve the action of the Dyer-Lashof algebra as it is a induced by a morphism of commutative $\HF_2$-algebras.
However, by assumption we have that some $x_{2^i-1}\in ker(\varphi_*)$ and so $\oo{x}_{2^i-1} \in ker \widetilde{\varphi}_*$.
For each $k$, $\oo{x}_{2^{i+k}-1}$ is a power operation on $\oo{x}_{2^i-1}$.
Therefore, each $\oo{x}_{2^{i+k}-1}$ must also be in the kernel, but they are not by construction of $I$.
\end{proof}

This should be compared with the results of Strickland, see \cite{StricklandMU}, where he realizes regular quotients of $MU_*$.
In his work, the ideal in question must satisfies a similar condition in order for the realization to be a commutative ring spectrum, see Proposition 6.2  in \cite{StricklandMU} for the relevant formula.

\subsection{Odd primary case}
\label{subsec: odd case}
When $p\neq 2$ the dual Steenrod algebra is no longer polynomial.
It is the tensor product of a polynomial algebra and an exterior algebra.
These exterior generators are what the $v_i$ detect.
In this subsection, we give the analogues of the above computations when $p\neq 2$.
First, we recall some facts about the dual Steenrod algebra and its structure as an algebra over the Dyer-Lashof algebra.

\begin{thm}[Milnor]
The dual Steenrod algebra is $\HH_*(\HF_p;\F_p) \iso \F_p [\xi_1, \xi_2, \xi_3, \ldots]\tensor \Lambda_{\F_p}(\tau_0,\tau_1,\tau_2,\ldots)$ (where $\Lambda$ denotes that the $\tau_i$'s are exterior generators).
The degree of $\vert \xi_i \vert=2p^i-2$ and $\vert \tau_i \vert=2p^i-1$.
The conjugation is given by 
\[
\oxi_i:=\chi_*(\xi_i)=\Sigma_{\alpha \in Part(i)}\Pi_{n=1}^{l(\alpha)}\xi^{p^{\sigma(i)}}_{\alpha(n)}
\]
where $l(\alpha)$ is the length of the ordered partition $\alpha$ of $i$ for all primes $p$.
\end{thm}
The map $\chi: \HF_p\wedge \HF_p \to \HF_p \wedge \HF_p$ is induced by the twist map.
We also need Steinberger's computation of the action of the Dyer-Lashof algebra on the dual Steenrod algebra.

\begin{thm}[Steinberger]
For all $p>2$, the dual Steenrod algebra is generated as an algebra over the Dyer-Lashof algebra by $\tau_0$.
The action is generated by the formulas
\begin{eqnarray*}
Q^{\rho(i)}\tau_0 &=& (-1)^i\oo{\tau}_i \\
\beta Q^{\rho(i)}\tau_0 &=& (-1)^i\oxi_i \\
\end{eqnarray*}
where $\rho(i):=\frac{p^i-1}{p-1}$.
In particular, we have that
\begin{eqnarray*}
Q^{p^{i-1}}(\oo{\tau}_{i-1}) &=& \oo{\tau}_i \\
\beta Q^{p^i}\oxi_i &=& \oxi_{i+1} \\
\end{eqnarray*}
for $i>0$.
\end{thm}

Before we proceed as above, we should mention that at an odd prime, $\BP$ is only known to be a commutative $S$-algebra at the primes $3$, which is due to \cite{HillLawson}.
Also, we work with the $p$-local Adams summand $\ell$ as well as $ku$.
They have different homotopy and homology as
\[
ku_{(p)}\we \bigvee\limits_{i=0}^{p-2}\sus^{2i} \ell.
\]
The computations in this section follow exactly the method carried out when $p=2$.
The difference between $\chi_*\tau_i$ and $-\tau_i$ is decomposable.
This follows as the we have that
\[
\tau_n +\Sigma_{i=0}^{n}\xi^{p^i}_{n-1}\chi_*\tau_i=0.
\]
While this means that the sign is not determined in the below formulas this will not be so concerning.
Our main use of these results is as obstructions, as we will see in \ref{sec:applications and interpretations}, and so we will see that their vanishing is what is most relevant.

\begin{prop}
The KSS $\Tor^{\ell_*}_* (\F_p,\F_p) \ssra \pi_*\HF_p \wedge_{\ell} \HF_p$ collapses at the $E_2$ page.
$\pi_* \HF_p \wedge_{\ell} \HF_p$ is an exterior algebra $E_{\F_p}[\oo{p},\oo{v}_1]$ with $Q^1(\oo{p})=\pm\oo{v}_1$ where $\vert \oo{p} \vert=1$ and $\vert \oo{v}_1 \vert=2p-1$.
\end{prop}
The proof of this follows the exact same lines as in the case for $ku$ and $p=2$, which is not surprising as $ku_{(2)} \we \ell$ at the prime $2$.
We have the following result regarding $ku$ at odd primes.

\begin{prop}
The KSS $\Tor^{ku_*}_* (\F_p,\F_p) \ssra \pi_*\HF_p \wedge_{ku} \HF_p$ collapses at the $E_2$ page.
$\pi_* \HF_p \wedge_{ku} \HF_p$ is an exterior algebra $E_{\F_p}[\oo{p},\oo{v}]$ where $\vert \oo{p} \vert=1$ and $\vert \oo{v} \vert=3$.
The Dyer-Lashof algebra acts trivially for degree reasons.
\end{prop}
The first possible nontrivial Dyer-Lashof operation is $Q^1(\oo{p})$, but this must be zero because $Q^1(\oo{p})\in \pi_{2p-1}(\HF_p \wedge_{ku} \HF_p)=0$.

\begin{prop}
The KSS $\Tor^{\BP_*}_* (\F_3,\F_3) \ssra \pi_*\HF_3 \wedge_{\BP} \HF_3$ collapses at the $E_2$ page.
$\pi_* \HF_3 \wedge_{\BP} \HF_3$ is an exterior algebra $E_{\F_3}[\oo{3},\oo{v}_1,\oo{v}_2]$ with $Q^1(\oo{3})=\pm\oo{v}_1$, $Q^4(\oo{3})=\pm\oo{v}_2$, $Q^3(\oo{v}_1)=\pm\oo{v}_2$, and $Q^4(\oo{3}\oo{v}_1)=\pm\oo{v}_1\oo{v}_2$ where $\vert \oo{3} \vert=1$, $\vert \oo{v}_1 \vert=5$ and $\vert \oo{v}_2 \vert=17$.
\end{prop}

Finally, we have the case of $MU$.
\begin{prop}
The KSS $\Tor^{MU_*}_* (\F_p,\F_p) \ssra \pi_*\HF_p \wedge_{MU} \HF_p$ collapses at the $E_2$ page.
$\pi_* \HF_p \wedge_{MU} \HF_p$ is an exterior algebra $E_{\F_p}[\oo{p},\oo{x}_1,\oo{x}_2, \ldots]$ with $Q^{\rho(i)}(\oo{p})=\pm\oo{x}_{p^{i-1}-1}$ and $Q^{p^i}(\oo{x}_{p^{i-1}-1})=\pm\oo{x}_{p^i-1}$ where $\vert \oo{p} \vert=1$ and $\vert \oo{x}_n \vert=2n+1$.
\end{prop}
And again, we have a corollary regarding complex orientations that lift to maps of commutative $S$-algebras.

\begin{cor}
\label{cor:strickland analog at p}
Let $I$ be an ideal of $MU_*$ generated by a regular sequence. 
If $I$ contains a finite nonzero number of the $x_{p^i-1}$, then the quotient map $MU \to MU/I$ cannot be realized as a map of commutative $S$-algebras.
\end{cor}

\section{Applications and Interpretations}
\label{sec:applications and interpretations}
In this section we present some applications and interpretations of the computations in section \ref{sec:comps}.
We must first develop ways of relating our computations to null-homotopies of maps.
Then we interpret our power operations as giving a dependence between choices of null-homotopies of different elements.
We first record some basic results relating classes in the KSS to differences of null-homotopies.
This difference is related to the ``suspension'' morphism which is studied in the setting of the B\"okstedt spectral sequence.
This suspension morphism is gotten by looking at the action of the fundamental class of $S^1$ on $S^1\tensor R$, however our relative smash products do not admit such a description.
When $R=C^*(X;\mathbb{Q})$ with $X$ simply connected then $S^1\tensor R \we C^*(LX;\mathbb{Q})$ while $\mathbb{Q}\wedge_R \mathbb{Q}\we C^*(\Omega X;\mathbb{Q})$ and so there is no circle action to utilize.
We begin with the following Theorem.

\begin{prop}
\label{prop: Tor-1 realization}
Suppose that $\varphi: R \to A$ is a map of commutative $S$-algebras that is surjective in homotopy.
Then $\forall x \in I:=ker(\varphi_*)$ with nonzero image in $I/I^2$ there is a nonzero class $\oo{x} \in \Tor^{R_*}_1(A_*,A_*)$.
If this class is not an ``eventual'' boundary in the KSS, then it can be realized as the difference of two null-homotopies of $\varphi_*(x) \in A_*$.
\end{prop}

\begin{proof}
Represent $x\in \pi_n R$ by a map $x: S^n \to R$.
Since $x\in I$, we know that the map $S^n \to R$ becomes null-homotopic when followed by $\varphi$.
Choose a null-homotopy and represent it as the following commutative diagram
\[
\xymatrix{
S^n \ar[rr]^x \ar[dd]_i&
&
R \ar[dd]_{\varphi}\\
&
&
\\
CS^n \ar[rr]^{x'}&
&
A.
}
\]
As $\varphi$ is onto in homotopy, we can take the first stage of the K\"unneth resolution (and filtration) to be $F_0=A_0=R$.
By assumption, $x\in \pi_*F_0= \F_0$ is in the kernel of the augmentation.
Therefore, $x \in im(\F_1 \to \F_0)$, so we get the following commutative diagram:

\[
\xymatrix{
S^n \ar[dd]^{\widetilde{x}} \ar[rr]^1&
&
S^n \ar[dd]^x \ar[rr]^i&
&
CS^n \ar[dd]_{x''}\\
&
&
&
&
\\
F_1 \ar[rr]^d&
&
F_0=R \ar[rr]^{\alpha_0}&
&
A_1.
}
\]
We can then suture the right side of the above diagram with the  diagram expressing the null-homotopy of $x$ along $x$ to produce
\[
\xymatrix{
S^n \ar[rrr] \ar[ddd] \ar[dr]^x&
&
&
CS^n \ar[dr]_{x'}\\
&
R \ar[rrr] \ar[ddd]&
&
&
A_1 \ar[ddd]\\
&
&
&
&
\\
CS^n \ar[dr]^{x''}&
&
&
&
\\
&
A \ar[rrr]&
&
&
A \wedge_R A_1
}
\]
as both $A$ and $A_1$ map to $A\wedge_R A_1$.
This induces a map from $S^{n+1}$ to $A \wedge_R A_1$.
Denote elements obtained in this way from $x\in \ker\varphi_*$ as $\oo{x} \in \pi_{n+1} A \wedge_R A_1$.
To see that this element contributes to 
\[
\ker(\pi_*(A)\tensor_{\R_*}\F_1\lra \pi_*(A)\tensor_{\R_*}\F_0),
\]
notice that the above definition of $\oo{x}$ gives the following commutative diagram

\[
\xymatrix{
CS^n \ar[dd] \ar[rr]&
&
S^{n+1} \ar[dd]^{\oo{x}} \ar[rr]^{\iso}&
&
S^{n+1} \ar[dd] \ar[rr]&
&
CS^{n+1} \ar[dd]
\\
&
&
&
&
&
&
\\
A=A \wedge_R A_0 \ar[rr]^{1_A \wedge_R \alpha_0} &
&
A \wedge_R A_1 \ar[rr]^{1_A \wedge_R \beta_1}&
&
A \wedge_R \sus F_1 \ar[rr]&
&
\sus A=A \wedge_R \sus F_0.
}
\]
We will now show that the image of $\oo{x}$ in $\pi_* A \tensor_{R_*} \F_1$ survives to a class in $\Tor^{R_*}_1(A_*,A_*)$.
Since $\varphi$ is surjective in homotopy, we have that $A_*\iso R_*/I$.
Applying $\Tor^{R_*}_*(A_*,-)$ to the short exact sequence
\[
0 \lra I \lra R_* \sr{\varphi_*}\lra A_* \lra 0
\]
gives a long exact sequence in $\Tor$.
This long exact sequence degenerates to the isomorphisms
\[
 \Tor^{R_*}_{n+1}(A_*,A_*)\iso \Tor^{R_*}_{n}(A_*,I).
\]
Therefore $\Tor^{R_*}_1(A_*,A_*)\iso A_* \tensor_{R_*} I$, which we compute by applying $-\tensor_{R_*} I$ to the right exact sequence
\[
I \to R_* \sr{\varphi_*}\lra A_* \lra 0
\]
which gives the exact sequence
\[
I \tensor_{R_*} I \sr{\mu}\lra I \lra A_* \tensor_{R_*} I \lra 0.
\]
Therefore, 
\[
\Tor^{R_*}_1(A_*,A_*) \iso A_*\tensor_{R_*} I\iso I/\mu(I) \iso I/I^2.
\]
Since we have assumed that $x$ has nonzero image in $I/I^2$ the element $\oo{x}$ is nonzero in $\Tor^{R_*}_1(A_*,A_*)$ as desired.
All that remains is that this class survives to contribute to $\pi_* A \wedge_R A$, which it does by assumption.
\end{proof}

We will apply the above result frequently to elements of $\pi_* A \wedge_R A$ that are detected on the $1$-line of the KSS.
Later in Proposition \ref{prop:main difference}, we will realize certain elements in $\pi_*A \wedge_R A$ as differences of null-homotopies.
The above result allows us to locate such elements in the KSS.
The next proposition is a similar result about how to relate power operations on relative smash products to the maps of commutative $S$-algebras.

\begin{prop}
Suppose that $\varphi:R \to R'$ is a map of commutative $S$-algebras over $\HF_p$, i.e. such that diagram of commutative $S$-algebras
\[
\xymatrix{
R\ar[rr]^{\varphi} \ar[dr]&
&
R'\ar[dl]\\
&
\HF_p&
}
\]
commutes.
If $x \in ker(\varphi_*)$ and $\oo{y}\in \{\ldots,\beta Q^i(\oo{x}), Q^i(\oo{x}),\ldots\} \subset \Tor^{R_*}_1(\F_p,\F_p)$, then $\varphi_*(y)\in \pi_*R'$ is decomposable.
\end{prop}
\begin{proof}
Since $\oo{y}=Q^i(\oo{x})$ and 
\[
\widehat{\varphi}_*:\Tor^{R_*}_s(\F_p,\F_p)_t \lra \Tor^{R'_*}_s(\F_p,\F_p)_t
\]
commutes with the action of these operations, we have that $\widehat{\varphi}_*(\oo{y})=0$ whenever $\widehat{\varphi}_*(\oo{x})=0$.
Since $\varphi_*(x)=0$, we have that $\widehat{\varphi}_*(\oo{x})=0\in \Tor^{R'_*}_1(\F_p,\F_p)$.
Therefore, $\widehat{\varphi}_*(\oo{y})=0\in\Tor^{R'_*}_1(\F_p,\F_p)\iso I/I^2$ where $I=ker(R'_* \to \F_p)$, and we have $\varphi_*(y)\in I^2$ as desired.
\end{proof}
Notice that this result does not require that the classes $\oo{x}$ or $\oo{y}$ survive the KSS and so it applies even when these classes may not give null-homotopies as they do in Proposition \ref{prop:main difference}.

The computations of section \ref{sec:comps} have explicit interpretations.
Let us work with a fixed map
\[
 \varphi: R \lra A
\]
of commutative $S$-algebras.
Further, we will assume that $A$ is a cofibrant $R$-module so that $A \wedge_R A$ is $A \wedge_R^L A$.
Here we will interpret the computation of the action of the $A$-Dyer-Lashof algebra on $\pi_* A \wedge_R A$.
At the end of this section we will apply our results to the example of Subsection \ref{subsec:ku over F_2} on $ku \to \HF_2$.

In commutative $S$-algebras, $A \wedge_R A$ is the pushout.
Therefore, a map of commutative $S$-algebras 
\[
\varphi : A\wedge_R A \to X
\]
is a map of commutative $S$-algebras from the diagram $A \lla R \lra A$ to $X$.
Thus, $\varphi$ specifies two potentially different commutative $A$-algebra structures on $X$ with the same underlying commutative $R$-algebra structure.
We wish to know, for example, how many different $A$-algebra structures on $X$ have the same underlying $R$-algebra structure.
How different are the $A$-algebra structures on $X$ coming from $f$ and $g$?
\begin{dfn}
We define the $d(f,g)$, the `\underline{difference}', of two maps of commutative $R$-algebras
\[
f,g: A \lra X
\]
to be the map they induce out of the pushout
\[
d(f,g): A\wedge_R A \lra X.
\]
\end{dfn}

The following lemma is illustrative of the role of the construction $d(f,g)$.
\begin{lemma}
The `difference' $d(f,f)=f\circ \mu$ and factors through the product map $A \wedge_R A \to A$.
\end{lemma}
In particular, this implies that we can detect the difference between $f$ and $g$ by considering the effect in homotopy of $d(f,g)$.
For example when $A$ is an Eilenberg-MacLane ring spectrum we see that $f \neq g$ whenever $d(f,g)$ is nonzero in homotopy outside degree zero.

\begin{prop}
\label{prop:main difference}
Let $x\in \ker{\varphi_*: R_* \to A_*}$ be such that the associated class $\oo{x} \in \Tor^{R_*}_1(A_*,A_*)$ survives the KSS and contributes $\oo{x} \in \pi_{n+1} A \wedge_R A$.
Then the class $\oo{x} \in \pi_{n+1} A \wedge_R A$ can be represented as a ``difference'' of null-homotopies of $x$.
\end{prop}
\begin{proof}
We can construct $\oo{x} \in \pi_{n+1} A \wedge_R A$ as a ``difference'' of null-homotopies as follows given that it comes from an element $x\in \ker{\varphi_*: R_* \to A_*}$.
When the map 
\[
x: S^n \lra R 
\]
is composed with
\[
\varphi: R \lra A
\]
it becomes null-homotopic.
After composing $\varphi\circ x$ with the right unit 
\[
\eta_r: A\we A\wedge_R R \sr{1\wedge \varphi}\lra A \wedge_R A
\]
we obtain the diagram
\[
\xymatrix{
S^n \ar[rr]^x \ar[dd]_i&
&
R \ar[dd]_{\varphi} \ar@/^1pc/[rddd]&
\\
&
&
&
\\
CS^n \ar[rr]^{x'} \ar@/_1pc/[rrrd]&
&
A \ar[dr]^{\eta_r}&
\\
&
&
&
A \wedge_R A.
}
\]
Composition with the left unit yields the diagram
 \[
\xymatrix{
S^n \ar[dd]^x \ar[rr]_i&
&
CS^n \ar[dd]^{x'} \ar@/^1pc/[rddd]&
\\
&
&
&
\\
R \ar[rr]_{\varphi}\ar@/_1pc/[rrrd]&
&
A \ar[dr]^{\eta_l}&
\\
&
&
&
A \wedge_R A.
}
\]
Suturing these diagrams together, we obtain
\[
\xymatrix{
S^n \ar[rrr]^i \ar[ddd]_i \ar[dr]^x&
&
&
CS^n \ar[dr]^{x'} \ar[ddd]\\
&
R \ar[rrr]^{\varphi} \ar[ddd]_{\varphi}&
&
&
A \ar[ddd]^{\eta_l}\\
&
&
&
&
\\
CS^n \ar[dr]_{x'} \ar[rrr]&
&
&
S^{n+1} \ar[dr]^{\oo{x}}&
\\
&
A \ar[rrr]^{\eta_r}&
&
&
A \wedge_R A
}
\]
where the back face is a pushout in $S$-modules. 
This constructs $\oo{x} \in \pi_{n+1} A \wedge_R A$.
\end{proof}
The class $\oo{x}$ is a witness to the fact that $x$ has been made null-homotopic in $A \wedge_R A$ in two different ways.
In fact, $d(f,g)_*(\oo{x})$ is the obstruction to $f$ and $g$ giving the same null-homotopy of $x$.

\begin{lemma}
\label{lemma:difference}
For $x \in \pi_n R$ as above, we have that $d(f,g)_*(\oo{x})=0$ in $\pi_{n+1}X$ if and only if $f \vert_{CS^n} \we g \vert_{CS^n}$ relative to $S^n$.
\end{lemma}
\begin{proof}
This proof is elementary homotopy theory.
Observe that $d(f,g)_*(\oo{x})=0$ in $\pi_{n+1} X$ if and only if it can be extended over $S^{n+1} \subset D^{n+1}$.
The element $d(f,g)_*(\oo{x})\pi_{n+1}(A\wedge_R A)$ was defined as $f \vert_{CS^n}$ on one half of $S^{n+1}$, $g \vert_{CS^n}$ on the other half of $S^{n+1}$, and $f \circ \varphi \circ x=g \circ \varphi \circ x$ on the equatorial $S^n$.
Therefore an extension of $d(f,g)_*(\oo{x})$ to all of $D^{n+1}$ is equivalent to the desired homotopy of $f \vert_{CS^n} \we g \vert_{CS^n}$ relative to $S^n$.
\end{proof}

We now apply the above discussion to interpret the computations made in the previous section.
Let $f,g: \HF_p \to X$ be two maps of commutative $S$-algebras  which are equalized by the reduction map $\HZ \to \HF_p$.
They then give a map $d(f,g): \HF_p \wedge_{\HZ} \HF_p \to X$.
The existence of $f$ and $g$ imply that we have a commutative $S$-algebra over $\HZ$ which can be made a commutative $\HF_p$-algebra in two potentially different ways, say $\widetilde{X}_f$ and $\widetilde{X}_g$.
By the above result, Lemma \ref{lemma:difference}, we see that if $d(f,g)_*(\oo{p}) \neq 0$ then $f$ is not homotopic to $g$.
It is in this way that we think of $\oo{p}$ as an obstruction.
It is not a complete obstruction in the sense that the vanishing of $d(f,g)_*(\oo{p})$  does not guarantee that $f\we g$, therefore we will refer to such classes as witnesses.

Now suppose in addition that we have a map of commutative $S$-algebras $\varphi : R \to \HF_p$ which also equalizes $f$ and $g$ and that $p \neq 0$ in $\pi_* R$.
Let $x \in \pi_* R$ be as above, so that we have $\oo{p},\oo{x} \in \pi_*\HF_p \wedge_R\HF_p$.
If $\oo{p}$ and $\oo{x}$ are related in $\pi_*\HF_p \wedge_R\HF_p$ by a power operation, say $Q^i(\oo{p})=\oo{x}$, then $d(f,g)_*(\oo{p})=0$ implies $d(f,g)_*(\oo{x})=0$ also.
In this way, $\oo{p}$ is a sort of indecomposable witness in that it generates an element in the cotangent complex of $\varphi$.
This simple observation allows us to interpret the above computations of section \ref{sec:comps} as saying that when we compare two different commutative $ku$, $\BP$, or $MU$-algebra structures on an $\Ei$-dga, the null-homotpies of $2$ that we choose effect the null-homotopies of $v$ and the $v_i$.
This seems a strange fact, as all of these classes necessarily act trivially on $X$ or $\widetilde{X}$.
However, their trivializations are related.
One might hope that the above considerations can be strengthened to give a construction of an explicit null-homotopy of $v$ or $v_i$ respectively.
For example. as $Q^{2}(\oo{2})=\oo{v}$ for $2,v \in \pi_* ku$ we would like to use a null-homotopy of $2$ to produce a null-homotopy of $v$ depending functorially on the null-homotopy of $2$.
The obvious construction fails as there is no operation $Q^2(2)=v$ in $\pi_* ku$, even though $Q^2(\oo{2})=\oo{v}$ in $\pi_* \HF_2 \wedge_{ku} \HF_2$.

\addcontentsline{toc}{section}{References}
\bibliography{Bibliography}{}

\begin{thebibliography}{10}

\bibitem{AngeltveitRognes}
Vigleik Angeltveit and John Rognes.
\newblock Hopf algebra structure on topological {H}ochschild homology.
\newblock {\em Algebr. Geom. Topol.}, 5:1223--1290, 2005.

\bibitem{BakerLazarev}
A.~Baker and A.~Lazarev.
\newblock On the {A}dams spectral sequence for {$R$}-modules.
\newblock {\em Algebr. Geom. Topol.}, 1:173--199, 2001.

\bibitem{BakerRichterAdams}
A.~Baker and B.~Richter.
\newblock On the cooperation algebra of the connective {A}dams summand.
\newblock {\em Tbil. Math. J.}, 1:33--70, 2008.

\bibitem{BrunerThesis}
R.~R. Bruner.
\newblock {\em T{HE} {ADAMS} {SPECTRAL} {SEQUENCE} {OF} {$H_{\infty}$} {RING}
  {SPECTRA}}.
\newblock ProQuest LLC, Ann Arbor, MI, 1977.
\newblock Thesis (Ph.D.)--The University of Chicago.

\bibitem{HRS}
R.~R. Bruner, J.~P. May, J.~E. McClure, and M.~Steinberger.
\newblock {\em {$H\sb \infty $} ring spectra and their applications}, volume
  1176 of {\em Lecture Notes in Mathematics}.
\newblock Springer-Verlag, Berlin, 1986.

\bibitem{DuggerMultSS}
D.~Dugger.
\newblock Multiplicative structures on homotopy spectral sequences i.
\newblock {\em http://http://pages.uoregon.edu/ddugger/mult.html}.

\bibitem{EKMM}
A.~D. Elmendorf, I.~Kriz, M.~A. Mandell, and J.~P. May.
\newblock {\em Rings, modules, and algebras in stable homotopy theory},
  volume~47 of {\em Mathematical Surveys and Monographs}.
\newblock American Mathematical Society, Providence, RI, 1997.
\newblock With an appendix by M. Cole.

\bibitem{HackneyOpsInfLoopSpace}
P.~Hackney.
\newblock Operations in the homology spectral sequence of a cosimplicial
  infinite loop space.
\newblock {\em J. Pure Appl. Algebra}, 217(7):1350--1377, 2013.

\bibitem{HillLawson}
M.~Hill and T.~Lawson.
\newblock Topological modular forms with level structure.
\newblock 2015.

\bibitem{HFpenKatthaenTilson}
Lukas Katth\"an and Sean Tilson.
\newblock The {H}omology of {C}onnective {M}orava {$E$}-theory with
  coefficients in {$\mathbb{F}_p$}, 2017.

\bibitem{LawsonNaumann}
T.~Lawson and N.~Naumann.
\newblock Commutativity conditions for truncated {B}rown-{P}eterson spectra of
  height 2.
\newblock {\em J. Topol.}, 5(1):137--168, 2012.

\bibitem{LigaardMadsen}
Hans Ligaard and Ib~Madsen.
\newblock Homology operations in the {E}ilenberg-{M}oore spectral sequence.
\newblock {\em Math. Z.}, 143:45--54, 1975.

\bibitem{LurieHA}
J.~Lurie.
\newblock Higher {A}lgebra.
\newblock \url{http://www.math.harvard.edu/~lurie/papers/HA.pdf}.

\bibitem{MaySteenrod}
J.~P. May.
\newblock A general algebraic approach to {S}teenrod operations.
\newblock In {\em The {S}teenrod {A}lgebra and its {A}pplications ({P}roc.
  {C}onf. to {C}elebrate {N}. {E}. {S}teenrod's {S}ixtieth {B}irthday,
  {B}attelle {M}emorial {I}nst., {C}olumbus, {O}hio, 1970)}, Lecture Notes in
  Mathematics, Vol. 168, pages 153--231. Springer, Berlin, 1970.

\bibitem{GreenBook}
Douglas~C. Ravenel.
\newblock {\em Complex cobordism and stable homotopy groups of spheres}, volume
  121 of {\em Pure and Applied Mathematics}.
\newblock Academic Press Inc., Orlando, FL, 1986.

\bibitem{StricklandMU}
N.~P. Strickland.
\newblock Products on {${\rm MU}$}-modules.
\newblock {\em Trans. Amer. Math. Soc.}, 351(7):2569--2606, 1999.

\bibitem{Tate}
J.~Tate.
\newblock Homology of {N}oetherian rings and local rings.
\newblock {\em Illinois J. Math.}, 1:14--27, 1957.

\bibitem{Weibel}
C.~Weibel.
\newblock {\em An introduction to homological algebra}, volume~38 of {\em
  Cambridge Studies in Advanced Mathematics}.
\newblock Cambridge University Press, Cambridge, 1994.

\end{thebibliography}
\bibliographystyle{plain}

\end{document}